\RequirePackage{amsthm}%

\documentclass[sn-mathphys,Numbered]{sn-jnl}
\usepackage{lmodern}%
\usepackage{graphicx}%
\usepackage{multirow}%
\usepackage{amsmath,amssymb,amsfonts,mathtools,bm}%
\usepackage{mathrsfs}%
\usepackage{enumitem}%
\usepackage[title]{appendix}%
\usepackage{xcolor}%
\usepackage{textcomp}%
\usepackage{manyfoot}%
\usepackage{booktabs}%
\usepackage{algorithm}%
\usepackage{algorithmicx}%
\usepackage{algpseudocode}%
\usepackage{listings}%
\usepackage{tikz}
\usetikzlibrary{positioning}
\usetikzlibrary{shapes.geometric}
\usepackage{svg}
\svgpath{{Images/}}
\graphicspath{{Images/}}

\usetikzlibrary{decorations.pathmorphing}\usetikzlibrary{patterns,arrows.meta,decorations.pathreplacing}

  \pgfdeclarepatternformonly{diagonalgray}
    {\pgfpoint{0pt}{0pt}}{\pgfpoint{6pt}{6pt}}{\pgfpoint{6pt}{6pt}}{
      \pgfsetlinewidth{0.5pt}
      \pgfpathmoveto{\pgfpoint{0pt}{0pt}}
      \pgfpathlineto{\pgfpoint{6pt}{6pt}}
      \pgfusepath{stroke}}

\newcommand{\N}{{\mathbb{N}}}

\newcommand{\R}{{\mathbb{R}}}

\newcommand{\Kc}{\mathcal{K}}
\newcommand{\Dc}{\mathcal{D}}
\newcommand{\Lc}{\mathcal{L}}

\newcommand{\Rc}{\mathcal{R}}
\newcommand{\Pc}{\mathcal{P}}

\newcommand{\Wc}{\mathcal{W}}
\newcommand{\Xc}{\mathcal{X}}
\newcommand{\Yc}{\mathcal{Y}}
\newcommand{\Zc}{\mathcal{Z}}

\DeclareMathOperator{\im}{im}

\DeclareMathOperator{\ext}{ext}

\newcommand{\pos}{\mathrm{pos}}
\newcommand{\pot}{\mathrm{pot}}
\newcommand{\kin}{\mathrm{kin}}

\newcommand{\setdef}[2]{\left\{ \, #1 \,\left\vert\vphantom{#1} \, #2 \, \right.\right\}}

\newenvironment{smallpmatrix}
  {\left(\begin{smallmatrix}}
  {\end{smallmatrix}\right)}
\newenvironment{smallbmatrix}
  {\left[\begin{smallmatrix}}
  {\end{smallmatrix}\right]}

\theoremstyle{thmstyleone}%
\newtheorem{theorem}{Theorem}
\newtheorem{proposition}[theorem]{Proposition}%
\newtheorem{lemma}[theorem]{Lemma}%

\theoremstyle{thmstyletwo}%
\newtheorem{example}{Example}%
\newtheorem{remark}{Remark}%

\theoremstyle{thmstylethree}%
\newtheorem{definition}{Definition}%

\newcommand{\bq}{\bm{\zeta}}
\newcommand{\bp}{\bm{\Gamma}}
\newcommand{\bv}{\bm{\varpi}}
\newcommand{\bF}{\bm{\tau}}

\begin{document}

\title[Rigid/flexible multibody systems]{Port-Hamiltonian modelling of coupled rigid/flexible multibody systems}


\author[1]{\fnm{Thomas} \sur{Berger}}\email{thomas.berger@math.upb.de}

\author[2]{\fnm{Ren\'e} \sur{Hochdahl}}\email{rene-christopher.hochdahl@tuhh.de}

\author*[3]{\fnm{Timo} \sur{Reis}}\email{timo.reis@tu-ilmenau.de}

\author[2]{\fnm{Robert} \sur{Seifried}}\email{robert.seifried@tuhh.de}


\equalcont{The authors contributed equally to this work.}

\affil[1]{\orgdiv{Institut f\"ur Mathematik}, \orgname{Martin-Luther-Universit\"at Halle-Wittenberg}, \orgaddress{\street{ Theodor-Lieser-Straße 5}, \city{Halle (Saale)}, \postcode{06099}, \country{Germany}}}

\affil[2]{\orgdiv{Institut f\"ur Mechanik und Meerestechnik}, \orgname{Technische Universit\"at Hamburg}, \orgaddress{\street{Ei\ss endorfer Stra\ss e 42}, \city{Hamburg}, \postcode{21073}, \country{Germany}}}

\affil*[3]{\orgdiv{Institut f\"ur Mathematik}, \orgname{Technische Universit\"at Ilmenau}, \orgaddress{\street{Weimarer Stra\ss e 25}, \city{Ilmenau}, \postcode{98693}, \country{Germany}}}

\abstract{
We develop a port-Hamiltonian framework for coupled rigid/flexible multibody systems. The rigid dynamics may be nonlinear and subject to configuration and velocity constraints, while the flexible components are described by linear port-Hamiltonian partial differential equations on one-dimensional spatial domains. The subsystems interact through boundary or distributed ports.

On the flexible side, the differential operator and its domain remain fixed, whereas state dependence enters only through finite-dimensional coupling components of the Dirac structure. We show that, under a natural surjectivity condition, the port-Hamiltonian interconnection with a modulated Dirac structure of a finite-dimensional rigid subsystem again yields a modulated Dirac structure. The Hamiltonians of the subsystems add, while the internal coupling powers cancel. The framework is illustrated by a planar moving Euler--Bernoulli beam and a slider--crank mechanism with a flexible connecting member.
}




\maketitle

\section{Introduction}
\label{sec:introduction}

Port-Hamiltonian systems provide a modular, energy-based framework for the
modeling of dynamical systems whose components exchange power through
well-defined ports. The framework applies to finite-dimensional mechanical
and electrical systems as well as to distributed-parameter systems; see, for
instance, \cite{JvdS14,vdS13,vdS17}. Algebraic constraints and implicit
energy-storage relations can be incorporated by means of port-Hamiltonian
differential--algebraic systems
\cite{vdS10,BMXZ18,MMW18,MvdS18,MvdS20,VvdS10a,GeHaRe20}, whereas
port-Hamiltonian formulations of partial differential equations have been
developed in geometric and operator-theoretic settings
\cite{MvdS04a,MvdS04b,MvdS02,Vill07,JZ12,ReSt21,R21,PhilReis23}.

In our preceding article \cite{BHRS25}, we developed a port-Hamiltonian
formulation for rigid multibody systems with position and velocity
constraints. The resulting models are nonlinear finite-dimensional differential--
algebraic systems. The present article extends this framework to
assemblies containing flexible components. Their deformation is described by
spatially distributed variables, so that the coupled dynamics take the form
of a DAE--PDE system.

Our aim is to develop a port-Hamiltonian interconnection framework for a broad class of coupled rigid/flexible multibody systems. On the rigid side, the framework accommodates nonlinear dynamics as well as configuration and velocity constraints. On the flexible side, it covers port-Hamiltonian partial differential equations whose differential operators and domains remain fixed. The subsystems may interact through boundary or distributed ports, with coupling laws that depend on their states. We show that the corresponding power-preserving interconnection is again represented by a modulated Dirac structure. Moreover, successive interconnections can be carried out within the same framework, which supports the modular construction of more complex mechanisms.

For the flexible subsystem, we use one-dimensional port-Hamiltonian
differential operators with boundary, distributed, and resistive ports in the
spirit of \cite{AugnerDis,Vill07,le2005dirac,JZ12}. State dependence is introduced through power-pairing-preserving
automorphisms acting only on selected finite-dimensional port variables. In this way, the
infinite-dimensional differential part remains unchanged, while moving-frame
or inertial effects are represented by finite-dimensional modulations of the
interconnection structure.

The main mechanical application is a planar Euler--Bernoulli beam moving with
both endpoints. Its deformation is expressed in a frame attached to the
straight segment joining the endpoints. A momentum transformation removes
the kinetic cross terms from the energy-storage relation, while the resulting
moving-frame terms are assigned to modulations of the
Dirac structures of the rigid and flexible subsystems. The model retains translational and rotational ports at
both endpoints and can therefore be embedded into larger multibody systems.
As an illustration a slider-crank mechanism is considered, where the beam is coupled to a rigid crank wheel and a straight
guide. Its left endpoint is attached to the crank pin, while its right
endpoint moves along the guide and acts as the slider. This construction uses
the associativity of Dirac interconnection and complements the fully rigid
slider--crank model considered in \cite[Sec.~7.3]{BHRS25}.

Flexible multibody dynamics has been studied extensively using classical and computational modeling approaches \cite{Shabana20,Simeon2013,CDP17}. First Port-Hamiltonian formulations range from flexible links and manipulators to floating-frame models of flexible multibody systems \cite{MacchelliMelchiorriStramigioli07,MattioniWuLeGorrec20,BAPM23}. Building on the constrained rigid multibody framework of \cite{BHRS25}, the present work focuses on the structural interconnection of rigid systems with infinite-dimensional flexible components at the continuous level. Analytical and numerical questions concerning the resulting DAE--PDE systems are beyond the scope of this article.

The article is organized as follows. Section~\ref{sec:pHsys} introduces the
port-Hamiltonian concepts used in the article and establishes the
interconnection result for modulated Dirac structures.
Section~\ref{sec:rigidmks} recalls the rigid multibody formulation from
\cite{BHRS25} in the form required for coupling with flexible components.
Section~\ref{sec:flexiblemks} introduces the fixed port-Hamiltonian PDE
structure and its finite-dimensional modulation. In Section~\ref{sec:ex}, we first derive the moving planar
Euler--Bernoulli beam and its port-Hamiltonian decomposition and then embed
this component into a slider--crank mechanism with a flexible connecting
member.

\subsection{Notation}

All vector spaces considered in this article are real. If $\Zc$ is a Hilbert
space, its inner product is denoted by
$\langle\cdot,\cdot\rangle_{\Zc}$; the subscript is omitted whenever the
underlying space is clear. The Euclidean inner product on $\R^n$ is written
as
$\langle z_1,z_2\rangle=z_1^\top z_2$.
For Hilbert spaces $\Xc$ and $\Yc$, the space of bounded linear operators from
$\Xc$ to $\Yc$ is denoted by $L(\Xc,\Yc)$, and we write
$L(\Xc):=L(\Xc,\Xc)$.
If $\mathcal S$ is a topological space and $\Wc$ and $\Yc$ are Hilbert
spaces, a mapping
$A:\mathcal S\to L(\Wc,\Yc)$
is called \emph{operator-norm continuous} if it is continuous when
$L(\Wc,\Yc)$ is equipped with the operator-norm topology. Thus no linear
structure on the parameter space $\mathcal S$ is required. 
$L(\Xc):=L(\Xc,\Xc)$.

We use the standard notation for Lebesgue and Sobolev spaces
\cite{Adam03}. For Hilbert-space-valued function spaces, the target space is
indicated after the domain; for example,
$L^2(\Omega;\Zc)$
denotes the space of square-integrable $\Zc$-valued functions on $\Omega$.
All integrals of Hilbert-space-valued functions are understood in the Bochner
sense; see \cite{Diestel77}.

We also use the standard notion of a continuously differentiable Hilbert
submanifold. Thus, a subset $\mathcal M$ of a Hilbert space $\Wc$ is a
Hilbert submanifold if, locally around every $x\in\mathcal M$, it can be
represented as
\[
\mathcal M\cap U
=
\setdef{y\in U}{f(y)=0},
\]
where $U\subset\Wc$ is open, $\Yc$ is a Hilbert space,
$f\in C^1(U;\Yc)$, and $f'(x):\Wc\to\Yc$ is surjective. Its tangent space is then
\[
T_x\mathcal M=\ker f'(x),
\]
or, equivalently, the set of velocities at $x$ of continuously
differentiable curves in $\mathcal M$. We refer to
\cite[Chap.~73]{Zeid88} for further details.


\section{Port-Hamiltonian systems and their interconnection}
\label{sec:pHsys}

\subsection{Port-Hamiltonian systems}
\label{sec:port-Hamiltonian-systems}

A central concept in port-Hamiltonian modeling is that of a Dirac structure.
It describes the power-preserving routing of energy between the storage,
dissipative, and external variables of a system. For an analytical treatment
of Dirac structures on Hilbert spaces, we refer to \cite{BHM23}.

Let $\Zc$ be a Hilbert space. On $\Zc\times\Zc$, consider the symmetric
bilinear form
\[
\begin{aligned}
\langle\!\langle\cdot,\cdot\rangle\!\rangle:
(\Zc\times\Zc)\times(\Zc\times\Zc)
&\to\R,
\\
\langle\!\langle(f_1,e_1),(f_2,e_2)\rangle\!\rangle
&:=
\langle f_1,e_2\rangle_{\Zc}
+
\langle f_2,e_1\rangle_{\Zc}.
\end{aligned}
\]
This form is called the \emph{power pairing}. For a subset
$\mathcal M\subset\Zc\times\Zc$, its \emph{power annihilator} is
\[
\mathcal M^{\bot\!\!\!\bot}
:=
\setdef{
(f,e)\in\Zc\times\Zc
}{
\langle\!\langle(f,e),(\widehat f,\widehat e)\rangle\!\rangle=0
\text{ for all }(\widehat f,\widehat e)\in\mathcal M
}.
\]
A subspace $\mathcal N\subset\Zc\times\Zc$ is called
\emph{power-isotropic} if
$\mathcal N\subset\mathcal N^{\bot\!\!\!\bot}$.
An element $z\in\Zc\times\Zc$ is called \emph{power-orthogonal} to a
subset $\mathcal M\subset\Zc\times\Zc$ if
\[
\langle\!\langle z,m\rangle\!\rangle=0
\qquad
\text{for all }m\in\mathcal M.
\]
Two subsets $\mathcal M,\mathcal N\subset\Zc\times\Zc$ are called
power-orthogonal if every element of $\mathcal M$ is power-orthogonal to
$\mathcal N$.

\begin{definition}[Dirac structure]
\label{def-Dir}
Let $\Zc$ be a Hilbert space. A subspace
$\mathcal D\subset\Zc\times\Zc$ is called a \emph{Dirac structure} if
$\mathcal D=\mathcal D^{\bot\!\!\!\bot}$.
\end{definition}

For $(f,e)\in\mathcal D$, the variables $f$ and $e$ are called the
\emph{flow} and the \emph{effort}, respectively. Every Dirac structure is
power-isotropic and closed, since every power annihilator is closed.

In a general bond-space formulation, flows and efforts belong to mutually
dual spaces. Since all spaces used in this article are Hilbert spaces, we
identify each space with its dual by means of the Riesz isomorphism. This
allows us to use the power pairing defined above. The distinction between
primal and dual spaces becomes essential for more general Stokes--Dirac
structures, in particular on higher-dimensional spatial domains; see
\cite{BHM23}.

In multibody systems, the power-preserving interconnection structure may
depend on the system state. This motivates the following notion.

\begin{definition}[Modulated Dirac structure]
\label{def-DirMod}
Let $\Xc$ be a topological space and let $\Zc$ be a Hilbert space. A family
$\mathcal D=(\mathcal D_x)_{x\in\Xc}$,
$\mathcal D_x\subset\Zc\times\Zc$,
is called a \emph{modulated Dirac structure} if the following conditions
hold:
\begin{enumerate}[label=(\alph*), ref=(\alph*)]
\item
\label{def-DirModa}
For every $x\in\Xc$, the subspace $\mathcal D_x$ is a Dirac structure.

\item
\label{def-DirModb}
There exists a Hilbert space $\Wc$ such that, for every $x\in\Xc$, there
exist an open neighborhood $U_x\subset\Xc$ and bounded linear operators
$T_y\in L(\Wc,\Zc\times\Zc)$, $y\in U_x$,
such that:
\begin{enumerate}[label=(\roman*)]
\item
$T_y$ is a bijection from $\Wc$ onto $\mathcal D_y$ for every
$y\in U_x$;
\item
the mapping
$U_x\to L(\Wc,\Zc\times\Zc)$,
$y\mapsto T_y$,
is continuous with respect to the operator norm.
\end{enumerate}
\end{enumerate}
\end{definition}
Condition~\ref{def-DirModb} is precisely what is meant by a local trivialization of the family $(\mathcal D_y)_{y\in\Xc}$. Since each $\mathcal D_y$ is closed, the inverse of the corresponding trivializing map is bounded by the bounded inverse theorem.

\begin{remark}
\label{rem:dirfin}
If $\dim\Zc=n<\infty$, then every Dirac structure
$\mathcal D_x\subset\Zc\times\Zc$ has dimension $n$; see
\cite[Sec.~2.2]{JvdS14}. Hence the model space in
Definition~\ref{def-DirMod} may be chosen as $\Wc=\R^n$.
\end{remark}

Energy storage is described by a Lagrangian submanifold. In a general
geometric formulation, this is a submanifold of a cotangent bundle. For the
systems considered here, it is sufficient to work in $\Zc\times\Zc$.

\begin{definition}[Lagrangian submanifold]
\label{def-Lag}
Let $\Zc$ be a Hilbert space. A Hilbert submanifold
$\mathcal L\subset\Zc\times\Zc$ is called a
\emph{Lagrangian submanifold} if, for every $z\in\mathcal L$,
\[
(v_1,v_2)\in T_z\mathcal L
\quad\Longleftrightarrow\quad
\langle v_1,w_2\rangle_{\Zc}
-
\langle w_1,v_2\rangle_{\Zc}
=0\;\text{ for all $(w_1,w_2)\in T_z\mathcal L$.}
\]
\end{definition}

For $\mathcal H\in C^2(U;\R)$ on an open finite-dimensional domain
$U$, the graph of $\nabla\mathcal H$ is a Lagrangian submanifold;
see \cite[Prop.~22.12]{Lee12}. The following elementary
infinite-dimensional version will be used for the flexible energy-storage
relation.

\begin{proposition}
\label{prop:infdimlagr}
Let $\Zc$ be a Hilbert space and let $H\in L(\Zc)$ be self-adjoint. Then
\[
\mathcal L_H
:=
\setdef{(x,Hx)}{x\in\Zc}
\subset\Zc\times\Zc
\]
is a Lagrangian submanifold.
\end{proposition}

\begin{proof}
The set $\mathcal L_H$ is a closed subspace and hence
$T_z\mathcal L_H=\mathcal L_H$ for every $z\in\mathcal L_H$.
For $x,e\in\Zc$, self-adjointness of $H$ gives
\[
\langle x,Hw\rangle_{\Zc}
-
\langle w,e\rangle_{\Zc}
=
\langle w,Hx-e\rangle_{\Zc}.
\]
This expression vanishes for all $w\in\Zc$ if and only if $e=Hx$.
Thus Definition~\ref{def-Lag} is satisfied.
\end{proof}

Dissipation is described by a resistive relation.

\begin{definition}[(Modulated) resistive relation]
\label{def-res}
Let $\Zc$ be a Hilbert space. A relation
$\mathcal R\subset\Zc\times\Zc$ is called \emph{resistive} if
\[
\langle f_\Rc,e_\Rc\rangle_{\Zc}\geq0
\qquad
\text{for every }(f_\Rc,e_\Rc)\in\mathcal R.
\]
If $\Xc$ is a set, a family
$\mathcal R=(\mathcal R_x)_{x\in\Xc}$
is called a \emph{modulated resistive relation} if every
$\mathcal R_x\subset\Zc\times\Zc$ is resistive.
\end{definition}

We can now define the class of port-Hamiltonian systems used in this article.
More general formulations on manifolds can be found in
\cite{MvdS20,JvdS14}.

\begin{definition}[Port-Hamiltonian system]
\label{def-pH}
Let $\Zc_\Lc$, $\Zc_\Rc$, and $\Zc_\Pc$ be Hilbert spaces, and let
$\Xc\subset\Zc_\Lc$ be open. A
\emph{port-Hamiltonian system on $\Xc$} is a differential inclusion of the
form
\begin{equation}
\left(
\begin{pmatrix}
\dot x(t)\\
f_\Rc(t)\\
f_\Pc(t)
\end{pmatrix},
\begin{pmatrix}
e_\Lc(t)\\
e_\Rc(t)\\
e_\Pc(t)
\end{pmatrix}
\right)
\in\mathcal D_{x(t)},
\qquad
\bigl(x(t),e_\Lc(t)\bigr)\in\mathcal L,
\qquad
\bigl(f_\Rc(t),e_\Rc(t)\bigr)\in\mathcal R_{x(t)},
\label{eq:phinc}
\end{equation}
where
\begin{itemize}
\item
$\mathcal D=(\mathcal D_x)_{x\in\Xc}$,
$\mathcal D_x
\subset
\bigl(\Zc_\Lc\times\Zc_\Rc\times\Zc_\Pc\bigr)^2$,
is a modulated Dirac structure;
\item
$\mathcal L\subset\Zc_\Lc\times\Zc_\Lc$
is a Lagrangian submanifold;
\item
$\mathcal R=(\mathcal R_x)_{x\in\Xc}$,
$\mathcal R_x\subset\Zc_\Rc\times\Zc_\Rc$,
is a modulated resistive relation.
\end{itemize}
The corresponding flow and effort variables are called
\emph{energy-storage variables}, \emph{resistive variables}, and
\emph{external port variables}, respectively.
\end{definition}

\begin{remark}[Sign convention]
The sign convention in Definition~\ref{def-pH} differs from the convention
used, for instance, in \cite{JvdS14}, where $-\dot x$ enters the Dirac
structure and the resistive relation is defined with the opposite sign. The
two conventions are equivalent after changing the signs of the corresponding
flow variables. With the convention used here,
$\langle f_\Pc,e_\Pc\rangle_{\Zc_\Pc}$
is the power delivered by the system through its external ports.
\end{remark}

\begin{remark}[Energy balance]
\label{rem:general-energy-balance}
Suppose that there exist an open set
$U\subset\Zc_\Lc$ containing the considered trajectory and a continuously
differentiable function
$\mathcal H:U\to\R$
such that
\[
\tfrac{\mathrm d}{\mathrm dt}\mathcal H(x(t))
=
\langle\dot x(t),e_\Lc(t)\rangle_{\Zc_\Lc}.
\]
The power-isotropy of $\mathcal D_{x(t)}$ then yields
\[
\tfrac{\mathrm d}{\mathrm dt}\mathcal H(x(t))
=
-
\langle f_\Rc(t),e_\Rc(t)\rangle_{\Zc_\Rc}
-
\langle f_\Pc(t),e_\Pc(t)\rangle_{\Zc_\Pc}.
\]
Consequently,
\[
\tfrac{\mathrm d}{\mathrm dt}\mathcal H(x(t))
\leq
-
\langle f_\Pc(t),e_\Pc(t)\rangle_{\Zc_\Pc}.
\]
Thus $-\langle f_\Pc,e_\Pc\rangle_{\Zc_\Pc}$ is the power supplied to the
system through its external ports.
\end{remark}

We finally fix the notation for products of relations. Let
$\Zc_1,\ldots,\Zc_k$ be Hilbert spaces and let
$\mathcal M_i\subset\Zc_i\times\Zc_i$,
$i=1,\ldots,k$.
Under the canonical identification of
$(\Zc_1\times\Zc_1)\times(\Zc_2\times\Zc_2)$
with
$(\Zc_1\times\Zc_2)\times(\Zc_1\times\Zc_2)$,
we write
\[
\mathcal M_1\times\mathcal M_2
:=
\setdef{
\left(
\begin{smallpmatrix}
f_1\\
f_2
\end{smallpmatrix},
\begin{smallpmatrix}
e_1\\[1mm]
e_2
\end{smallpmatrix}
\right)
}{
(f_1,e_1)\in\mathcal M_1,\quad
(f_2,e_2)\in\mathcal M_2
}.
\]
Products of three or more relations are defined associatively. The same
notation is used pointwise for families of relations.

The following properties follow directly from the definitions:
\begin{enumerate}[label=(\roman*)]
\item
the product of Dirac structures is a Dirac structure, and the product of
modulated Dirac structures is a modulated Dirac structure;

\item
the product of Lagrangian submanifolds is a Lagrangian submanifold;

\item
the product of resistive relations is a resistive relation, and the
corresponding statement holds for modulated resistive relations.
\end{enumerate}

\subsection{Interconnection of port-Hamiltonian systems}
\label{sec:interc}

A fundamental feature of the port-Hamiltonian framework is its modularity
under power-preserving interconnections; see, for instance,
\cite{BCGM18,CvdSB07,JvdS14,VvdS10b}. For each subsystem, the external port
is decomposed into a port that remains external and a coupling port that is
connected to another subsystem. If the coupling variables
$(f_{{\rm c}1},e_{{\rm c}1})$ and
$(f_{{\rm c}2},e_{{\rm c}2})$
take values in the same Hilbert space $\Zc_{\rm c}$, the interconnection is
defined by
\begin{equation}
f_{{\rm c}1}=f_{{\rm c}2},
\qquad
e_{{\rm c}1}=-e_{{\rm c}2}.\label{eq:couplcond}
\end{equation}
Hence,
\[
\langle f_{{\rm c}1},e_{{\rm c}1}\rangle_{\Zc_{\rm c}}
+
\langle f_{{\rm c}2},e_{{\rm c}2}\rangle_{\Zc_{\rm c}}
=0,
\]
so that the coupling ports exchange power internally without generating or
dissipating it. In
mechanical applications, equality of the coupling flows typically expresses
kinematic compatibility, whereas the opposite coupling efforts express force
equilibrium. The interconnection is depicted in
Figure~\ref{fig:pHinterc}.

\begin{figure}
    \centering
    \begin{tikzpicture}[scale=0.5]
\usetikzlibrary{patterns}
\usetikzlibrary{calc}
\pgfdeclarelayer{background}
\pgfsetlayers{background,main}

\coordinate (D1) at (0,0);
\coordinate (H1) at (-3,4);
\coordinate (R1) at (3,4);

\coordinate (a11) at ($(D1)+(0.3,0.2)$);
\coordinate (a12) at ($(H1)+(0.3,0.2)$);
\draw[line width=3pt] (a11) -- (a12);
\node at ($0.5*(a11)+0.5*(a12)+(0.5,0.3)$) {$e_{\mathcal{S}1}$};

\coordinate (b11) at ($(D1)+(-0.3,-0.2)$);
\coordinate (b12) at ($(H1)+(-0.3,-0.2)$);
\draw[line width=3pt] (b11) -- (b12);
\node at ($0.5*(b11)+0.5*(b12)+(-0.7,-0.0)$) {$f_{\mathcal{S}1}$};

\coordinate (c11) at ($(D1)+(0.3,-0.2)$);
\coordinate (c12) at ($(R1)+(0.3,-0.2)$);
\draw[line width=3pt] (c11) -- (c12);
\node at ($0.5*(c11)+0.5*(c12)+(0.7,-0.0)$) {$e_{\mathcal{R}1}$};

\coordinate (d11) at ($(D1)+(-0.3,0.2)$);
\coordinate (d12) at ($(R1)+(-0.3,0.2)$);
\draw[line width=3pt] (d11) -- (d12);
\node at ($0.5*(d11)+0.5*(d12)+(-0.5,0.3)$) {$f_{\mathcal{R}1}$};

\coordinate (e11) at ($(D1)+(0.3,-0.2)$);
\coordinate (e12) at ($(D1)+(0.3,-3.2)$);
\draw[line width=3pt] (e11) -- (e12);
\node at ($0.5*(e11)+0.5*(e12)+(0.6,-0.96)$) {$e_{\mathcal{P}1}$};

\coordinate (f11) at ($(D1)+(-0.3,-0.2)$);
\coordinate (f12) at ($(D1)+(-0.3,-3.2)$);
\draw[line width=3pt] (f11) -- (f12);
\node at ($0.5*(f11)+0.5*(f12)+(-0.6,-0.9)$) {$f_{\mathcal{P}1}$};

\newcommand{\opaqueellipse}[4]{
  \fill[white] (#1) ellipse [x radius=#2, y radius=#3];
  \fill[pattern=diagonalgray, pattern color=gray!50] (#1) ellipse [x radius=#2, y radius=#3];
  \draw[thick] (#1) ellipse [x radius=#2, y radius=#3];
  \node at (#1) {\Large {#4}};
}

\opaqueellipse{D1}{3cm}{2cm}{$\Dc_1$}
\opaqueellipse{H1}{1.8cm}{1.2cm}{$\mathcal{H}_1$}
\opaqueellipse{R1}{1.8cm}{1.2cm}{$\mathcal{R}_1$}

\begin{scope}[xshift=11cm]  

\coordinate (D2) at (0,0);
\coordinate (H2) at (-3,4);
\coordinate (R2) at (3,4);

\coordinate (a21) at ($(D2)+(0.3,0.2)$);
\coordinate (a22) at ($(H2)+(0.3,0.2)$);
\draw[line width=3pt] (a21) -- (a22);
\node at ($0.5*(a21)+0.5*(a22)+(0.5,0.3)$) {$e_{\mathcal{S}2}$};

\coordinate (b21) at ($(D2)+(-0.3,-0.2)$);
\coordinate (b22) at ($(H2)+(-0.3,-0.2)$);
\draw[line width=3pt] (b21) -- (b22);
\node at ($0.5*(b21)+0.5*(b22)+(-0.7,-0.0)$) {$f_{\mathcal{S}2}$};

\coordinate (c21) at ($(D2)+(0.3,-0.2)$);
\coordinate (c22) at ($(R2)+(0.3,-0.2)$);
\draw[line width=3pt] (c21) -- (c22);
\node at ($0.5*(c21)+0.5*(c22)+(0.7,-0.0)$) {$e_{\mathcal{R}2}$};

\coordinate (d21) at ($(D2)+(-0.3,0.2)$);
\coordinate (d22) at ($(R2)+(-0.3,0.2)$);
\draw[line width=3pt] (d21) -- (d22);
\node at ($0.5*(d21)+0.5*(d22)+(-0.5,0.3)$) {$f_{\mathcal{R}2}$};

\coordinate (e21) at ($(D2)+(0.3,-0.2)$);
\coordinate (e22) at ($(D2)+(0.3,-3.2)$);
\draw[line width=3pt] (e21) -- (e22);
\node at ($0.5*(e21)+0.5*(e22)+(0.6,-0.96)$) {$e_{\mathcal{P}2}$};

\coordinate (f21) at ($(D2)+(-0.3,-0.2)$);
\coordinate (f22) at ($(D2)+(-0.3,-3.2)$);
\draw[line width=3pt] (f21) -- (f22);
\node at ($0.5*(f21)+0.5*(f22)+(-0.6,-0.9)$) {$f_{\mathcal{P}2}$};

\opaqueellipse{D2}{3cm}{2cm}{$\Dc_2$}
\opaqueellipse{H2}{1.8cm}{1.2cm}{$\mathcal{H}_2$}
\opaqueellipse{R2}{1.8cm}{1.2cm}{$\mathcal{R}_2$}

\end{scope}

\begin{pgfonlayer}{background}  
  \coordinate (a1) at ($(D1)+(0.3,0.2)$);
  \coordinate (a2) at ($(D2)+(-0.3,0.2)$);
  \draw[line width=3pt] (a1) -- (a2);
  \node[align=left, anchor=west] at ($0.5*(a1)+0.5*(a2)+(-1.4,0.5)$) {$f_{\mathrm{c}1}=-f_{\mathrm{c}2}$};

  \coordinate (b1) at ($(D1)+(0.3,-0.2)$);
  \coordinate (b2) at ($(D2)+(-0.3,-0.2)$);
  \draw[line width=3pt] (b1) -- (b2);
  \node[align=left, anchor=west] at ($0.5*(b1)+0.5*(b2)+(-1.4,-0.5)$) {$e_{\mathrm{c}1}=\phantom{-}e_{\mathrm{c}2}$};
\end{pgfonlayer}
\end{tikzpicture}
    \caption{Interconnection of two port-Hamiltonian systems.}
    \label{fig:pHinterc}
\end{figure}
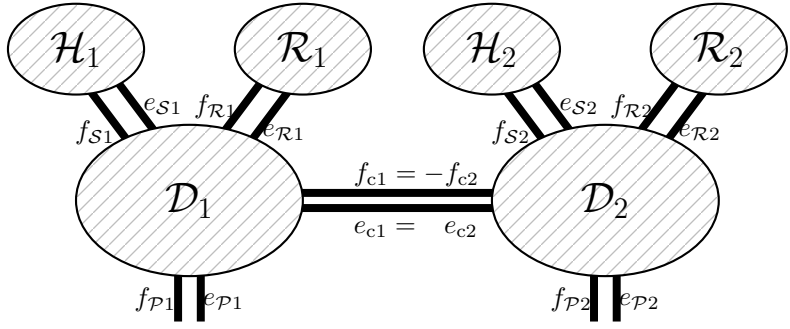

\begin{definition}[Interconnection of Dirac structures]
\label{def:DiracInterc}
Let $\Xc_1$ and $\Xc_2$ be topological spaces, and let
$\Zc_1$, $\Zc_2$, and $\Zc_{\rm c}$ be Hilbert spaces. Moreover, let
\[
\Dc_i=(\Dc_{i,x_i})_{x_i\in\Xc_i},
\qquad i=1,2,
\]
be families of Dirac structures satisfying
\[
\Dc_{i,x_i}
\subset
(\Zc_i\times\Zc_{\rm c})^2.
\]
For $(x_1,x_2)\in\Xc_1\times\Xc_2$, their
\emph{interconnection with respect to $\Zc_{\rm c}$} is
\[
\Dc_{1,x_1}\circ_{\Zc_{\rm c}}\Dc_{2,x_2}
:=
\setdef{
\begin{aligned}\left(
\begin{pmatrix}
f_1\\
f_2
\end{pmatrix},
\begin{pmatrix}
e_1\\
e_2
\end{pmatrix}
\right)\\
\in(\Zc_1\times\Zc_2)^2
\end{aligned}}{
\begin{aligned}
\text{there exist }f_{\rm c},e_{\rm c}\in\Zc_{\rm c}
\text{ such that}
\\
\left(
\begin{smallpmatrix}
f_1\\
f_{\rm c}
\end{smallpmatrix},
\begin{smallpmatrix}
e_1\\[1mm]
e_{\rm c}
\end{smallpmatrix}
\right)
\in\Dc_{1,x_1},\\
\text{and }\left(
\begin{smallpmatrix}
f_2\\
f_{\rm c}
\end{smallpmatrix},
\begin{smallpmatrix}
e_2\\[1mm]
-e_{\rm c}
\end{smallpmatrix}
\right)
\in\Dc_{2,x_2}
\end{aligned}}.
\]
The corresponding family is denoted by
\[
\Dc_1\circ_{\Zc_{\rm c}}\Dc_2
:=
\bigl(
\Dc_{1,x_1}\circ_{\Zc_{\rm c}}\Dc_{2,x_2}
\bigr)_{(x_1,x_2)\in\Xc_1\times\Xc_2}.
\]
The same notation will be used for families of arbitrary subspaces in place
of families of Dirac structures.
\end{definition}

\begin{remark}[Associativity of interconnection]
\label{rem:order-of-interconnection}
Let
\[
\mathcal M_1
\subset
\bigl(\Zc_1\times\Zc_{{\rm c}12}\bigr)^2,\quad
\mathcal M_2
\subset
\bigl(\Zc_2\times\Zc_{{\rm c}12}\times\Zc_{{\rm c}23}\bigr)^2,\quad
\mathcal M_3
\subset
\bigl(\Zc_3\times\Zc_{{\rm c}23}\bigr)^2
\]
be subspaces. Then, up to the canonical ordering of the remaining flow and
effort variables,
\[
\bigl(
\mathcal M_1\circ_{\Zc_{{\rm c}12}}\mathcal M_2
\bigr)
\circ_{\Zc_{{\rm c}23}}\mathcal M_3
=
\mathcal M_1\circ_{\Zc_{{\rm c}12}}
\bigl(
\mathcal M_2\circ_{\Zc_{{\rm c}23}}\mathcal M_3
\bigr).
\]
Indeed, both sides impose the same coupling conditions and eliminate the
same internal port variables. The identity is therefore purely algebraic
and also holds pointwise for families of subspaces. For finite-dimensional
Dirac structures, see also \cite{CvdSB07}.
\end{remark}

The following elementary lemma will be used to construct local
trivializations of interconnected families.

\begin{lemma}[Continuous families of kernels]
\label{lem:continuous-kernels}
Let $\Xc$ be a topological space, let $\Wc$ and $\Yc$ be Hilbert spaces, and
let
$A_x\in L(\Wc,\Yc)$,
$x\in\Xc$,
be an operator-norm continuous family of surjective operators. For every
$x^0\in\Xc$, there exist an open neighborhood $U$ of $x^0$ and an
operator-norm continuous family
\[
L_x:\ker A_{x^0}\to\ker A_x,
\qquad x\in U,
\]
of bounded linear isomorphisms.
\end{lemma}

\begin{proof}
Let
$\Kc:=\ker A_{x^0}$
and choose a bounded linear right inverse
$R\in L(\Yc,\Wc)$
of $A_{x^0}$. Then
$\Wc=\Kc\oplus\im R$.
Since $A_{x^0}R=I_{\Yc}$, the operator $A_xR$ is invertible for all $x$ in
a sufficiently small neighborhood $U$ of $x^0$. For $x\in U$, define
\[
L_xu
:=
u-R(A_xR)^{-1}A_xu,
\qquad
u\in\Kc.
\]
Then $A_xL_xu=0$, so that
$\im L_x\subset\ker A_x$.
Conversely, if $v\in\ker A_x$, write
\[
v=u+Ry,
\qquad
u\in\Kc,\quad y\in\Yc.
\]
The identity $A_xv=0$ gives
\[
y=-(A_xR)^{-1}A_xu,
\]
and hence $v=L_xu$. Thus $\im L_x=\ker A_x$. Moreover, $L_x$ is injective
because
$\Kc\cap\im R=\{0\}$.
The operator-norm continuity of $x\mapsto L_x$ follows directly from its
definition.
\end{proof}

\begin{proposition}[Interconnection through a finite-dimensional coupling space]
\label{prop:Dirac_comp}
Let $\Xc_1$ and $\Xc_2$ be topological spaces, let
$\Zc_1$ and $\Zc_2$ be Hilbert spaces, and let
$\Zc_{\rm c}$ be a finite-dimensional Hilbert space. Let
$\Dc_i
=
(\Dc_{i,x_i})_{x_i\in\Xc_i}$,
$i=1,2$,
be modulated Dirac structures satisfying
$\Dc_{i,x_i}
\subset
(\Zc_i\times\Zc_{\rm c})^2$.
For $(x_1,x_2)\in\Xc_1\times\Xc_2$, define
\[
\begin{aligned}
\Gamma_{x_1,x_2}:
\Dc_{1,x_1}\times\Dc_{2,x_2}
&\to
\Zc_{\rm c}\times\Zc_{\rm c},
\\
(d_1,d_2)
&\mapsto
\begin{pmatrix}
f_{{\rm c}1}-f_{{\rm c}2}\\
e_{{\rm c}1}+e_{{\rm c}2}
\end{pmatrix},
\end{aligned}
\]
where $f_{{\rm c}i}$ and $e_{{\rm c}i}$ denote the coupling components of
$d_i$. Assume that $\Gamma_{x_1,x_2}$ is surjective for every
$(x_1,x_2)\in\Xc_1\times\Xc_2$. Then
$\Dc_1\circ_{\Zc_{\rm c}}\Dc_2$
is a modulated Dirac structure.
\end{proposition}

\begin{proof}
We first verify the pointwise Dirac property. Fix
$x=(x_1,x_2)\in\Xc_1\times\Xc_2$.
With respect to the decompositions
$\Zc_i\times\Zc_{\rm c}$, the spaces
$\Dc_{i,x_i}$ are split Dirac structures in the terminology of
\cite{BKvdSZ10}. After the fixed
power-pairing-preserving sign change
\[
(f_{{\rm c}2},e_{{\rm c}2})
\mapsto
(-f_{{\rm c}2},-e_{{\rm c}2}),
\]
the coupling conditions \eqref{eq:couplcond} agree with the composition
conditions used in \cite[Cor.~3.9(ii)]{BKvdSZ10}. Since
$\Zc_{\rm c}\times\Zc_{\rm c}$ is finite-dimensional, that result implies
that
\[
\Dc_{1,x_1}\circ_{\Zc_{\rm c}}\Dc_{2,x_2}
\]
is a Dirac structure.

It remains to construct local trivializations. For
$i=1,2$, set
$\mathfrak B_i
:=
(\Zc_i\times\Zc_{\rm c})^2$.
Fix
$x^0=(x_1^0,x_2^0)\in\Xc_1\times\Xc_2$.
By Definition~\ref{def-DirMod}, there exist Hilbert spaces
$\Wc_1,\Wc_2$, open neighborhoods
$U_i\subset\Xc_i$ of $x_i^0$, and operator-norm continuous families of
bounded linear isomorphisms
\[
T_{i,y_i}:\Wc_i\to\Dc_{i,y_i},
\qquad
y_i\in U_i.
\]
Set
$\Wc:=\Wc_1\times\Wc_2$,
$\Yc:=\Zc_{\rm c}\times\Zc_{\rm c}$,
and, for $y=(y_1,y_2)\in U_1\times U_2$, define
$A_y\in L(\Wc,\mathfrak B_1\times\mathfrak B_2)$
by
\[
A_y(u_1,u_2)
:=
\bigl(T_{1,y_1}u_1,T_{2,y_2}u_2\bigr).
\]
Further, define
\[
\widehat\Gamma_y
:=
\Gamma_{y_1,y_2}A_y
\in L(\Wc,\Yc).
\]
The family
$y\mapsto\widehat\Gamma_y$
is operator-norm continuous, and every
$\widehat\Gamma_y$ is surjective.

All kernels of surjective bounded operators from $\Wc$ onto the
finite-dimensional space $\Yc$ are closed subspaces of $\Wc$ of codimension
$\dim\Yc$ and are therefore mutually isomorphic. Fix once and for all a Hilbert space
$\Kc$ representing this isomorphism class and choose a bounded linear
isomorphism
\[
J_{x^0}:
\Kc\to\ker\widehat\Gamma_{x^0}.
\]
By Lemma~\ref{lem:continuous-kernels}, after shrinking
$U_1\times U_2$ if necessary, there exists an operator-norm continuous
family of bounded linear isomorphisms
\[
L_y:
\ker\widehat\Gamma_{x^0}
\to
\ker\widehat\Gamma_y.
\]
Let
$\mathcal Q:
\mathfrak B_1\times\mathfrak B_2
\to
(\Zc_1\times\Zc_2)^2$
be the projection that removes the coupling components, and define
\[
\Theta_y
:=
\mathcal QA_yL_yJ_{x^0}.
\]
By construction,
\[
\Theta_y:
\Kc
\to
\Dc_{1,y_1}\circ_{\Zc_{\rm c}}\Dc_{2,y_2}
\]
is bounded and surjective.

To prove injectivity, suppose that $\Theta_yu=0$ and set
\[
(d_1,d_2):=A_yL_yJ_{x^0}u.
\]
Since $(d_1,d_2)\in\ker\Gamma_{y_1,y_2}$ and all non-coupling components
vanish, there exist
$f_{\rm c},e_{\rm c}\in\Zc_{\rm c}$ such that
\[
d_1
=
\left(
\begin{pmatrix}
0\\
f_{\rm c}
\end{pmatrix},
\begin{pmatrix}
0\\
e_{\rm c}
\end{pmatrix}
\right),
\qquad
d_2
=
\left(
\begin{pmatrix}
0\\
f_{\rm c}
\end{pmatrix},
\begin{pmatrix}
0\\
-e_{\rm c}
\end{pmatrix}
\right).
\]
Let $a,b\in\Zc_{\rm c}$ be arbitrary. The surjectivity of
$\Gamma_{y_1,y_2}$ yields
$\widehat d_i\in\Dc_{i,y_i}$ such that
\[
\widehat f_{{\rm c}1}-\widehat f_{{\rm c}2}=a,
\qquad
\widehat e_{{\rm c}1}+\widehat e_{{\rm c}2}=b.
\]
Using the power-isotropy of $\Dc_{1,y_1}$ and $\Dc_{2,y_2}$ gives
\[
0
=
\langle f_{\rm c},b\rangle_{\Zc_{\rm c}}
+
\langle a,e_{\rm c}\rangle_{\Zc_{\rm c}}.
\]
Since $a$ and $b$ are arbitrary,
$f_{\rm c}=e_{\rm c}=0$.
Hence $d_1=d_2=0$, and the injectivity of
$A_y$, $L_y$, and $J_{x^0}$ implies $u=0$.

Thus $\Theta_y$ is a bounded linear bijection onto the interconnected
Dirac structure. Since the latter is closed, the bounded inverse theorem
shows that $\Theta_y$ is a bounded linear isomorphism. Moreover,
$y\mapsto\Theta_y$ is operator-norm continuous. The maps $\Theta_y$
therefore form a local trivialization with the fixed model space $\Kc$.
Hence
$\Dc_1\circ_{\Zc_{\rm c}}\Dc_2$
is a modulated Dirac structure.
\end{proof}
We now apply Proposition~\ref{prop:Dirac_comp} to port-Hamiltonian systems.
Let
\[
(\Dc_i,\Lc_i,\Rc_i),
\qquad i=1,2,
\]
be two port-Hamiltonian systems, and suppose that the external port space of
subsystem $i$ is decomposed as
\[
\Zc_{\Pc_i}^{\rm ext}\times\Zc_{\rm c},
\]
where $\Zc_{\rm c}$ is finite-dimensional, the first factor contains the
ports that remain external, and the second factor contains the coupling
ports. Set
\[
\Zc_i
:=
\Zc_{\Lc_i}\times\Zc_{\Rc_i}\times\Zc_{\Pc_i}^{\rm ext}.
\]
Then
\[
\Dc_{i,x_i}\subset(\Zc_i\times\Zc_{\rm c})^2.
\]
If the coupling operators from Proposition~\ref{prop:Dirac_comp} are
surjective, the interconnected system is governed by
\[
\Lc=\Lc_1\times\Lc_2,
\qquad
\Rc=\Rc_1\times\Rc_2.
\]
Let
\[
\Pi:
(\Zc_1\times\Zc_2)^2
\to
\bigl(
\Zc_{\Lc_1}\times\Zc_{\Lc_2}
\times\Zc_{\Rc_1}\times\Zc_{\Rc_2}
\times\Zc_{\Pc_1}^{\rm ext}\times\Zc_{\Pc_2}^{\rm ext}
\bigr)^2
\]
denote the fixed permutation that groups the energy-storage, resistive, and
remaining external variables. Then
\[
\Dc_{(x_1,x_2)}
=
\Pi\bigl(
\Dc_{1,x_1}\circ_{\Zc_{\rm c}}\Dc_{2,x_2}
\bigr).
\]
Since $\Pi$ preserves the power pairing, $\Dc$ is a modulated Dirac
structure. Thus the interconnected system is again port-Hamiltonian.

If the subsystems possess storage functions $\mathcal H_1$ and
$\mathcal H_2$, then the interconnected storage function is
\[
\mathcal H(x_1,x_2)
=
\mathcal H_1(x_1)+\mathcal H_2(x_2).
\]
The powers at the internal coupling ports cancel, so that
Remark~\ref{rem:general-energy-balance} applies with only the remaining
external ports.

We next introduce the particular class of state-dependent
Dirac structures of flexible subsystems used later in the article.

\begin{definition}[Finite-dimensional modulation]
\label{def:finite-dimensional-modulation}
Let $\Xc$ be a topological space, and let $\Zc$ and
$\Zc_{\rm c}$ be Hilbert spaces. Set
\[
\mathfrak B
:=
(\Zc\times\Zc_{\rm c})^2.
\]
A family
\[
\Dc=(\Dc_x)_{x\in\Xc},
\qquad
\Dc_x\subset\mathfrak B,
\]
is said to have a \emph{finite-dimensional modulation relative to
$\Zc_{\rm c}$} if there exist an integer $r\in\N_0$, a closed subspace
$\mathcal T\subset\mathfrak B$, a finite-dimensional subspace
$\mathcal W\subset\mathfrak B$, and subspaces
$\mathcal E_x\subset\mathcal W$, $x\in\Xc$, such that:
\begin{enumerate}[label=(\alph*), ref=(\alph*)]
\item
every element of $\mathcal T$ has vanishing coupling-port components, that is,
\[
\mathcal T
\subset
\setdef{
\left(
\begin{pmatrix}
f\\
0
\end{pmatrix},
\begin{pmatrix}
e\\
0
\end{pmatrix}
\right)
}{
f,e\in\Zc
};
\]

\item
$\mathcal T$ is power-isotropic and
\[
\mathcal T^{\bot\!\!\!\bot}
=
\mathcal T\oplus\mathcal W,
\qquad
\dim\mathcal W=2r;
\]

\item
for every $x\in\Xc$, the subspace $\mathcal E_x$ is
self-orthogonal in $\mathcal W$ with respect to the restricted power
pairing, that is,
\[
\mathcal E_x
=
\setdef{
w\in\mathcal W
}{
\langle\!\langle w,v\rangle\!\rangle=0
\text{ for all }v\in\mathcal E_x
};
\]

\item
for every $x\in\Xc$,
\[
\Dc_x
=
\mathcal T\oplus\mathcal E_x;
\]

\item
for every $x\in\Xc$, there exist an open neighborhood
$U_x\subset\Xc$ and linear isomorphisms
\[
S_y:\R^r\to\mathcal E_y,
\qquad
y\in U_x,
\]
such that
\[
U_x\to L(\R^r,\mathfrak B),
\qquad
y\mapsto S_y,
\]
is operator-norm continuous.
\end{enumerate}
The integer $r$ is called the \emph{defect index} of the family.
\end{definition}

\begin{lemma}
\label{lem:finite-dimensional-modulation}
Every family satisfying
Definition~\ref{def:finite-dimensional-modulation} is a modulated Dirac
structure.
\end{lemma}

\begin{proof}
Fix $x\in\Xc$. Since
$\mathcal W\subset\mathcal T^{\bot\!\!\!\bot}$, the subspaces
$\mathcal T$ and $\mathcal W$, and hence also
$\mathcal T$ and $\mathcal E_x$, are power-orthogonal. Moreover,
$\mathcal T$ is power-isotropic by assumption, and the self-orthogonality
of $\mathcal E_x$ in $\mathcal W$ implies that $\mathcal E_x$ is
power-isotropic. Consequently,
\[
\Dc_x
=
\mathcal T\oplus\mathcal E_x
\subset
\Dc_x^{\bot\!\!\!\bot}.
\]

Conversely, let
$z\in\Dc_x^{\bot\!\!\!\bot}$.
Since $\mathcal T\subset\Dc_x$, we have
\[
z\in\mathcal T^{\bot\!\!\!\bot}
=
\mathcal T\oplus\mathcal W.
\]
Write
$z=z_{\mathcal T}+w$ with
$z_{\mathcal T}\in\mathcal T$ and $w\in\mathcal W$.
The inclusion
$z\in\Dc_x^{\bot\!\!\!\bot}$ implies that $z$ is power-orthogonal to
$\mathcal E_x$. Since $\mathcal T$ and $\mathcal E_x$ are
power-orthogonal, it follows that
\[
\langle\!\langle w,v\rangle\!\rangle=0
\qquad
\text{for every }v\in\mathcal E_x.
\]
The self-orthogonality of $\mathcal E_x$ in $\mathcal W$ therefore gives
$w\in\mathcal E_x$. Hence
$z\in\mathcal T\oplus\mathcal E_x=\Dc_x$, and thus
\[
\Dc_x=\Dc_x^{\bot\!\!\!\bot}.
\]

For $y\in U_x$, define
\[
A_y:\mathcal T\times\R^r\to\Dc_y,
\qquad
(z_{\mathcal T},v)\mapsto z_{\mathcal T}+S_yv.
\]
The direct-sum representation
$\Dc_y=\mathcal T\oplus\mathcal E_y$
shows that $A_y$ is a bounded linear bijection. Since $\Dc_y$ is closed,
the bounded inverse theorem implies that $A_y$ is an isomorphism. The
operator-norm continuity of $y\mapsto A_y$ follows from that of
$y\mapsto S_y$. Hence the maps $A_y$ provide the required local
trivializations.
\end{proof}

The following construction provides a simple criterion for producing
modulated Dirac structures.

\begin{proposition}[Modulation by power-preserving automorphisms]
\label{prop:power-preserving-modulation}
Let $\Xc$ be a topological space, let $\Zc$ be a Hilbert space, and let
$\Dc^0=(\Dc_x^0)_{x\in\Xc}$,
$\Dc_x^0\subset\Zc\times\Zc$,
be a modulated Dirac structure. Suppose that
$U_x\in L(\Zc\times\Zc)$,
$x\in\Xc$,
is an operator-norm continuous family of bounded linear automorphisms preserving the power pairing, that is,
\[
\langle\!\langle U_xz_1,U_xz_2\rangle\!\rangle
=
\langle\!\langle z_1,z_2\rangle\!\rangle
\qquad
\text{for all }z_1,z_2\in\Zc\times\Zc.
\]
Then
$\Dc=(\Dc_x)_{x\in\Xc}$,
$\Dc_x:=U_x\Dc_x^0$,
is a modulated Dirac structure.
\end{proposition}

\begin{proof}
Since $U_x$ preserves the power pairing,
$(U_x\mathcal M)^{\bot\!\!\!\bot}
=
U_x\bigl(\mathcal M^{\bot\!\!\!\bot}\bigr)$
for every subspace $\mathcal M\subset\Zc\times\Zc$. Hence every
$\Dc_x$ is a Dirac structure. If
$T_y:\Wc\to\Dc_y^0$ is a local trivialization of $\Dc^0$, then
$U_yT_y:\Wc\to\Dc_y$
is a local trivialization of $\Dc$. Its operator-norm continuity follows
from that of $y\mapsto U_y$ and $y\mapsto T_y$.
\end{proof}

\begin{remark}[Several finite-dimensional components]
\label{rem:several-finite-dimensional-subsystems}
Suppose that several finite-dimensional modulated Dirac structures,
including static interconnection or termination structures, have been
combined by products and interconnections while a finite-dimensional
coupling space $\Zc_{\rm c}$ remains external. If the resulting modulated Dirac structure of the rigid subsystem
\[
\Dc_{\rm r}
=
(\Dc_{{\rm r},x_{\rm r}})_{x_{\rm r}\in\Xc_{\rm r}},
\qquad
\Dc_{{\rm r},x_{\rm r}}
\subset
(\Zc_{\rm r}\times\Zc_{\rm c})^2,
\]
and the modulated Dirac structure of the flexible subsystem
satisfy the hypotheses of
Proposition~\ref{prop:Dirac_comp}, then their interconnection is again a
modulated Dirac structure. By
Remark~\ref{rem:order-of-interconnection}, the resulting relation agrees,
up to the ordering of its components, with the relation obtained by
interconnecting all subsystems simultaneously.
\end{remark}

\section{Rigid multibody systems}
\label{sec:rigidmks}

We recall the port-Hamiltonian formulation of rigid multibody systems
developed in \cite{BHRS25}. The presentation below contains the ingredients
needed for the subsequent coupling with flexible components. Alternative
formulations based directly on configuration manifolds can be found, for
instance, in \cite{Arn89,Duindam2009}.

We use configuration variables
$\bq\in U_{\pos}\subset\R^{n_\pot}$,
generalized velocity variables
$\bv\in\R^{n_\kin}$,
generalized momentum variables
$\bp\in\R^{n_\kin}$,
and generalized force variables in $\R^{n_\kin}$.
The dimensions $n_\pot$ and $n_\kin$ need not coincide. In particular, the
configuration variables may be redundant, and dependencies between them are
retained as algebraic constraints rather than eliminated by a
parametrization.

The model data are as follows. The potential energy is described by
$\mathcal V_\pot\in C^2(U_{\pos};\R)$,
and the kinematic relation between generalized velocities and configuration rates by a continuous mapping
$Z:U_{\pos}\to\R^{n_\pot\times n_\kin}$.
The kinetic energy is
$\frac12\bv^\top M\bv$,
where $M\in\R^{n_\kin\times n_\kin}$ is symmetric. In mechanical
applications, $M$ is typically positive semidefinite and, in non-degenerate
models, positive definite. 
Gyroscopic terms are
represented by a continuous mapping
$G:\R^{n_\kin}\to\R^{n_\kin\times n_\kin}$ that takes values in the set of skew-symmetric matrices. 
Let
$U_{\rm d}\subset U_{\pos}\times\R^{n_\kin}$.
The damping force
$\bF_{\rm d}:U_{\rm d}\to\R^{n_\kin}$
is assumed to satisfy
$\bv^\top\bF_{\rm d}(\bq,\bv)\geq0$ for all $(\bq,\bv)\in U_{\rm d}$.
Configuration and velocity constraints are described by
$c\in C^2(U_{\pos};\R^k)$,
$A\in C(U_{\pos};\R^{\ell\times n_\kin})$,
through
$c(\bq)=0$ and
$A(\bq)\bv=0$.
We assume that $c'(\bq)$ and $A(\bq)$ have full row rank for every
$\bq\in U_{\pos}$.
Finally, the continuous mappings
$B_{\ext}:U_{\pos}\to\R^{n_\kin\times m_{\ext}}$,
$B_{\rm c}:U_{\pos}\to\R^{n_\kin\times m_{\rm c}}$
describe the external and coupling ports, respectively.

The rigid multibody model is the differential--algebraic system
\begin{equation}
\begin{aligned}
\dot{\bq}
&=
Z(\bq)\bv,
\\[1mm]
M\dot{\bv}
&=
-Z(\bq)^\top
\bigl(
\nabla\mathcal V_\pot(\bq)
+
c'(\bq)^\top\bm\lambda
\bigr)
-\bF_{\rm d}(\bq,\bv)
-G(M\bv)\bv
\\
&\quad
-A(\bq)^\top\bm\mu
-B_{\ext}(\bq)\bF_{\ext}
-B_{\rm c}(\bq)\bF_{\rm c},
\\[1mm]
0
&=
c(\bq),
\\
0
&=
A(\bq)\bv,
\\
\bv_{\ext}
&=
B_{\ext}(\bq)^\top\bv,
\\
\bv_{\rm c}
&=
B_{\rm c}(\bq)^\top\bv.
\end{aligned}
\label{eq:mks2}
\end{equation}
Here $\bm\lambda$ and $\bm\mu$ are the Lagrange multipliers associated with
the configuration and velocity constraints. The variables $\bv_{\ext}$ and
$\bv_{\rm c}$ are co-located with the generalized port efforts
$\bF_{\ext}$ and $\bF_{\rm c}$.

In accordance with the sign convention of
Definition~\ref{def-pH}, the products
$\bF_{\ext}^\top\bv_{\ext}$,
$\bF_{\rm c}^\top\bv_{\rm c}$
represent power delivered by the rigid subsystem through the corresponding
ports. Thus, if $\bF_{\ext}^{\rm act}$ and $\bF_{\rm c}^{\rm act}$ denote
forces acting on the rigid subsystem, then
\[
\bF_{\ext}=-\bF_{\ext}^{\rm act},
\qquad
\bF_{\rm c}=-\bF_{\rm c}^{\rm act}.
\]

We next recall the port-Hamiltonian representation of
\eqref{eq:mks2}. Its state consists of the configuration variables and generalized momenta,
\[
x_{\rm r}
=
\begin{pmatrix}
\bq\\
\bp
\end{pmatrix},
\qquad
\bp=M\bv.
\]
The Dirac structure is the family
$\Dc_{\rm r}
=
\bigl(
\Dc_{{\rm r},(\bq,\bp)}
\bigr)_{
(\bq,\bp)\in U_{\pos}\times\R^{n_\kin}
}$ with flow and effort variables ordered as
\[
f_{\rm r}
=
\begin{pmatrix}
\bv_{\Lc,f}\\
\bF_{\Lc,f}\\
\bv_{\Rc}\\
\bv_{\ext}\\
\bv_{\rm c}
\end{pmatrix},
\qquad
e_{\rm r}
=
\begin{pmatrix}
\bF_{\Lc,e}\\
\bv_{\Lc,e}\\
\bF_{\Rc}\\
\bF_{\ext}\\
\bF_{\rm c}
\end{pmatrix}.
\]
For $(\bq,\bp)\in U_{\pos}\times\R^{n_\kin}$, define
\begin{equation}
\Dc_{{\rm r},(\bq,\bp)}:=\setdef{ \parbox{3.9cm}{$\displaystyle\left(\begin{pmatrix}\bv_{\Lc,f}\\\bF_{\Lc,f}\\\bv_{\Rc}\\\bv_{\ext}\\\bv_{\rm c}\end{pmatrix},\begin{pmatrix}\bF_{\Lc,e}\\\bv_{\Lc,e}\\\bF_{\Rc}\\\bF_{\ext}\\\bF_{\rm c}\end{pmatrix}\right)$\\[3mm] $\in(\R^{n_\pot+2n_\kin+m_{\ext}+m_{\rm c}})^2$}}{\,
{\parbox{6cm}{
$\bv_{\Lc,f}=Z(\bq)\bv_{\Lc,e},\;\bv_{\Rc}=\bv_{\Lc,e},$\\[1mm] $A(\bq) \bv_{\Lc,e}=0$,\; $\bv_{\ext}=B_{\ext}(\bq)^\top \bv_{\Lc,e}$,\\[1mm] $\bv_{\rm c}=B_{\rm c}(\bq)^\top \bv_{\Lc,e}$,\\[2mm] $\exists\,\bm{\mu}\in\R^\ell:$\\
$\bF_{\Lc,f}+Z(\bq)^\top\bF_{\Lc,e}+G(\bp)\bv_{\Lc,e}+\bF_{\Rc}$\\$+B_{\ext}(\bq)\bF_{\ext}+B_{\rm c}(\bq)\bF_{\rm c}+A(\bq)^\top \bm{\mu}=0$}}\!\!}.\label{eq:rigidDirac2}\end{equation}
By \cite[Prop.~2]{BHRS25}, $\Dc_{\rm r}$ is a finite-dimensional modulated
Dirac structure.

The resistive relation is
\begin{equation}
\begin{aligned}
\Rc_{\rm r}
&=
\bigl(
\Rc_{{\rm r},(\bq,\bp)}
\bigr)_{
(\bq,\bp)\in U_{\pos}\times\R^{n_\kin}
},\\
\Rc_{{\rm r},(\bq,\bp)}
&:=
\setdef{
(\bv,\bF)\in(\R^{n_\kin})^2
}{
(\bq,\bv)\in U_{\rm d},
\quad
\bF=\bF_{\rm d}(\bq,\bv)
}.
\end{aligned}
\label{eq:MKSresrel2}
\end{equation}
It is independent of $\bp$ and is resistive by the assumed inequality
$\bv^\top\bF_{\rm d}(\bq,\bv)\geq0$.
The energy-storage relation is
\begin{equation}
\Lc_{\rm r}
:=
\setdef{
\left(
\begin{pmatrix}
\bq\\
\bp
\end{pmatrix},
\begin{pmatrix}
\bF\\
\bv
\end{pmatrix}
\right)
\in
(\R^{n_\pot+n_\kin})^2
}{
\begin{aligned}
&\bq\in U_{\pos},
\qquad
M\bv=\bp,
\qquad
c(\bq)=0,
\\
&\exists\,\bm\lambda\in\R^k:
\quad
\bF
=
\nabla\mathcal V_\pot(\bq)
+
c'(\bq)^\top\bm\lambda
\end{aligned}
}.
\label{eq:LagrMech}
\end{equation}
It is shown in \cite[Prop.~5]{BHRS25} that $\Lc_{\rm r}$ is a Lagrangian
submanifold.

With
\[
\Zc_{\rm r}
:=
\R^{n_\pot+n_\kin}
\times
\R^{n_\kin}
\times
\R^{m_{\ext}}
\qquad\text{and}\qquad
\Zc_{\rm c}
:=
\R^{m_{\rm c}},
\]
the Dirac structure of the rigid subsystem satisfies $\Dc_{{\rm r},(\bq,\bp)}
\subset
(\Zc_{\rm r}\times\Zc_{\rm c})^2$.
It therefore has precisely the form required for the first subsystem in
Proposition~\ref{prop:Dirac_comp}.

\section{Flexible components}
\label{sec:flexiblemks}

\subsection{A fixed port-Hamiltonian PDE structure}
\label{sec:fixed-flexible-structure}
We first introduce a fixed port-Hamiltonian PDE structure for flexible
components on a one-dimensional spatial domain, following
\cite{AugnerDis,Vill07,le2005dirac,JZ12}. The class includes, among others,
vibrating strings as well as Euler--Bernoulli and Timoshenko beams. The
state dependence required for the coupled systems considered below will be
confined to finitely many port variables; the differential operator and its
domain remain fixed.

Throughout this section, ports are written in the generator convention.
Thus the duality product of flow and effort is the power delivered by the
subsystem to its environment. For mechanical ports, flows are velocities or
angular velocities, whereas efforts have the sign of reaction forces or
reaction moments, that is, the opposite sign of forces or moments acting on
the subsystem.

Let $N,d\in\N$, let $P_0,\ldots,P_N\in\R^{d\times d}$ satisfy
$P_k=(-1)^{k+1}P_k^\top$ for $k=0,\ldots,N$, and assume that $P_N$ is
invertible. Further, let
$D,H\in L^\infty([a,b];\R^{d\times d})$, where $D+D^\top$ is pointwise
positive semidefinite, and $H$ is pointwise symmetric and positive definite
with $H^{-1}\in L^\infty([a,b];\R^{d\times d})$. Finally, let
$B_0\in L^2([a,b];\R^{d\times m_{\rm dis}})$, where
$m_{\rm dis}\in\N_0$.

The corresponding reference system is
\begin{equation}
\begin{aligned}
\tfrac{\partial x}{\partial t}(\xi,t)
&=
-D(\xi)H(\xi)x(\xi,t)
+
\sum_{k=0}^N
P_k\tfrac{\partial^k}{\partial\xi^k}
\bigl(H(\xi)x(\xi,t)\bigr)
-
B_0(\xi)\widehat e_{\rm dis}(t),
\\
\widehat f_{\rm dis}(t)
&=
\int_a^b
B_0(\xi)^\top H(\xi)x(\xi,t)\,{\rm d}\xi .
\end{aligned}
\label{eq:BH1gen}
\end{equation}
Here $\widehat e_{\rm dis}$ is the effort acting through the fixed
distributed profile $B_0$, and $\widehat f_{\rm dis}$ is its co-located flow.
We write $x(t):=x(\cdot,t)$ for the spatial state.

The boundary variables are defined by
\begin{equation}
f_\partial(t)
=
W_f\gamma\bigl(Hx(t)\bigr),
\qquad
e_\partial(t)
=
W_e\gamma\bigl(Hx(t)\bigr),
\label{eq:flixBnd}
\end{equation}
where
\begin{equation}
\gamma:
H^N([a,b];\R^d)\to\R^{2Nd},
\ z\mapsto
\begin{pmatrix}
z(b)\\
\vdots\\
z^{(N-1)}(b)\\
z(a)\\
\vdots\\
z^{(N-1)}(a)
\end{pmatrix}
\label{eq:BNDTrace}
\end{equation}
is the boundary trace operator. We choose $W_f,W_e\in\R^{Nd\times 2Nd}$ such that
$W:=\begin{smallbmatrix}W_f\\W_e\end{smallbmatrix}$ is invertible and,
with
$\Sigma:=\begin{smallbmatrix}0&I_{Nd}\\I_{Nd}&0\end{smallbmatrix}$,
the Green identity
\begin{equation}
\left\langle
z,
\sum_{k=0}^N P_k z^{(k)}
\right\rangle_{L^2([a,b];\R^d)}
=
-\frac12\,
\gamma(z)^\top W^\top\Sigma W\gamma(z)
\qquad
\text{for all }z\in H^N([a,b];\R^d)
\label{eq:BNDtriplet}
\end{equation}
holds. For
$f_\partial=W_f\gamma(z)$ and
$e_\partial=W_e\gamma(z)$,
the right-hand side of \eqref{eq:BNDtriplet} equals
$-f_\partial^\top e_\partial$.
Such boundary configurations exist; see
\cite[Lem.~3.4 and Def.~3.5]{le2005dirac}.
The resulting boundary relation is a Dirac structure by
\cite[Thm.~3.6]{le2005dirac}; see also
\cite[Thm.~2.7]{Vill07}. In applications, $W_f$ and $W_e$ may be chosen
directly in terms of the physical boundary velocities, forces, angular
velocities, and moments.

\begin{proposition}
\label{Prop:Diracflex0}
Under the assumptions above, the relation
\begin{equation}
\Dc_\partial
:=
\setdef{
\begin{aligned}
&\;\;\left(
\begin{pmatrix}
f\\
f_\partial
\end{pmatrix},
\begin{pmatrix}
e\\
e_\partial
\end{pmatrix}
\right)\\
&\in
\bigl(
L^2([a,b];\R^d)\times\R^{Nd}
\bigr)^2
\end{aligned}}{
\begin{aligned}
&e\in H^N([a,b];\R^d),
\\
&f=\sum_{k=0}^N P_ke^{(k)},\;
\begin{pmatrix}
f_\partial\\
e_\partial
\end{pmatrix}
=
W\gamma(e)
\end{aligned}
}
\label{eq:boundary-dirac}
\end{equation}
is a Dirac structure.
\end{proposition}

\begin{proof}
This is the boundary Dirac structure from
\cite[Thm.~3.6]{le2005dirac}; see also
\cite[Thm.~2.7]{Vill07}, adapted to the generator sign convention used
here. Indeed, \eqref{eq:BNDtriplet} is the corresponding Green identity, while
the invertibility of $W$ ensures that the boundary traces are parametrized
by the flow and effort variables.
\end{proof}

We next record a simple augmentation principle that allows bounded internal
and distributed ports to be appended to a Dirac structure.

\begin{lemma}
\label{lem:Diracext}
Let $\Zc$ and $\Zc_0$ be Hilbert spaces, let
$\Dc\subset\Zc\times\Zc$ be a Dirac structure, and let
$C\in L(\Zc,\Zc_0)$. Then
\[
\Dc_{\rm aug}
:=
\setdef{
\left(
\begin{pmatrix}
f-C^*e_0\\
Ce
\end{pmatrix},
\begin{pmatrix}
e\\
e_0
\end{pmatrix}
\right)
}{
(f,e)\in\Dc,\quad e_0\in\Zc_0
}
\subset
(\Zc\times\Zc_0)^2
\]
is a Dirac structure.
\end{lemma}

\begin{proof}
For $(f,e),(\widehat f,\widehat e)\in\Dc$ and
$e_0,\widehat e_0\in\Zc_0$, the additional terms in the power pairing
cancel:
\[
\left\langle
\begin{pmatrix}
\widehat f-C^*\widehat e_0\\
C\widehat e
\end{pmatrix},
\begin{pmatrix}
e\\
e_0
\end{pmatrix}
\right\rangle
+
\left\langle
\begin{pmatrix}
f-C^*e_0\\
Ce
\end{pmatrix},
\begin{pmatrix}
\widehat e\\
\widehat e_0
\end{pmatrix}
\right\rangle
=
\langle\widehat f,e\rangle
+
\langle f,\widehat e\rangle
=
0.
\]
Thus $\Dc_{\rm aug}$ is power-isotropic.

Conversely, let
\[
\left(
\begin{pmatrix}
g_1\\
g_0
\end{pmatrix},
\begin{pmatrix}
h_1\\
h_0
\end{pmatrix}
\right)
\in
\Dc_{\rm aug}^{\bot\!\!\!\bot}.
\]
Testing against all elements of $\Dc_{\rm aug}$ gives
\[
\langle f,h_1\rangle
+
\langle g_1+C^*h_0,e\rangle
+
\langle g_0-Ch_1,e_0\rangle
=
0
\]
for all $(f,e)\in\Dc$ and all $e_0\in\Zc_0$. Hence
\[
(g_1+C^*h_0,h_1)\in\Dc^{\bot\!\!\!\bot}=\Dc,
\qquad
g_0=Ch_1.
\]
It follows directly that the element belongs to $\Dc_{\rm aug}$.
\end{proof}

Let $J\subset[a,b]$ be measurable and assume that
$D(\xi)=0$ for almost every $\xi\in[a,b]\setminus J$.
For a Hilbert space $\Xc$, let
${\rm Restr}_J:L^2([a,b];\Xc)\to L^2(J;\Xc)$
denote the restriction operator. Its adjoint $({\rm Restr}_J)^*$ is the extension-by-zero operator. One may, in particular, choose
$J=\operatorname{ess\,supp}D$.

\begin{proposition}
\label{Prop:Diracflex}
Under the assumptions above, the relation
\begin{multline}
\widehat{\Dc}
=
\setdef{
\left(
\begin{pmatrix}
f\\
f_\Rc\\
\widehat f_{\rm dis}\\
f_\partial
\end{pmatrix},
\begin{pmatrix}
e\\
e_\Rc\\
\widehat e_{\rm dis}\\
e_\partial
\end{pmatrix}
\right)
}{
\begin{aligned}
&e\in H^N([a,b];\R^d),
\\
&f
=
\sum_{k=0}^N P_ke^{(k)}
-
({\rm Restr}_J)^*e_\Rc
-
B_0\widehat e_{\rm dis},
\\
&f_\Rc={\rm Restr}_J e,
\\
&\widehat f_{\rm dis}
=
\int_a^b B_0(\xi)^\top e(\xi)\,{\rm d}\xi,
\\
&\begin{pmatrix}
f_\partial\\
e_\partial
\end{pmatrix}
=
W\gamma(e)
\end{aligned}
}
\\
\subset
\bigl(
L^2([a,b];\R^d)
\times L^2(J;\R^d)
\times\R^{m_{\rm dis}}
\times\R^{Nd}
\bigr)^2
\label{eq:flexDirac}
\end{multline}
is a Dirac structure.
\end{proposition}

\begin{proof}
Apply Lemma~\ref{lem:Diracext} to the Dirac structure
$\Dc_\partial$ from Proposition~\ref{Prop:Diracflex0} and the bounded
operator
\[
C:
L^2([a,b];\R^d)\times\R^{Nd}
\to
L^2(J;\R^d)\times\R^{m_{\rm dis}},
\ (e,e_\partial) \mapsto
\begin{pmatrix}
{\rm Restr}_J e\\[1mm]
\displaystyle
\int_a^b B_0(\xi)^\top e(\xi)\,{\rm d}\xi
\end{pmatrix}.
\]
After a fixed reordering of the flow and effort variables, the augmented
relation is precisely \eqref{eq:flexDirac}.
\end{proof}

The energy-storage relation is
\begin{equation}
\Lc_{\rm f}
=
\setdef{
(x,Hx)
}{
x\in L^2([a,b];\R^d)
}.
\label{eq:flexLagr}
\end{equation}
Since the multiplication operator induced by $H$ is bounded and
self-adjoint, Proposition~\ref{prop:infdimlagr} implies that
$\Lc_{\rm f}$ is a Lagrangian submanifold.

The resistive relation is
\begin{equation}
\Rc_{\rm f}
=
\setdef{
(f_\Rc,D|_Jf_\Rc)
}{
f_\Rc\in L^2(J;\R^d)
}.
\label{eq:flexRes}
\end{equation}
It is resistive since
\[
\langle f_\Rc,D|_Jf_\Rc\rangle
=
\tfrac12
\langle
f_\Rc,(D+D^\top)|_Jf_\Rc
\rangle
\geq0.
\]
The port-Hamiltonian system governed by
$\widehat{\Dc}$, $\Lc_{\rm f}$, and $\Rc_{\rm f}$ is precisely the reference
system \eqref{eq:BH1gen} with boundary ports \eqref{eq:flixBnd}.

\subsection{Finite-dimensional modulation of flexible port variables}
\label{sec:flexible-port-modulation}

The Dirac structure $\widehat{\Dc}$ from
Proposition~\ref{Prop:Diracflex} describes a fixed differential operator
together with its resistive, distributed, and boundary ports. We now
introduce state dependence through a power-preserving transformation of
finitely many port variables.

Let $\Zc_0$ be a Hilbert space, let $\Zc_{\rm m}$ be finite-dimensional,
and let
$\widehat{\Dc}
\subset
(\Zc_0\times\Zc_{\rm m})^2$
be a fixed Dirac structure. The space $\Zc_0$ contains the distributed
variables and all port variables that are not modulated, whereas
$\Zc_{\rm m}$ contains the finite-dimensional port variables on which the
modulation acts.

Let $\Xc_{\rm f}$ be a topological space and let
$K:\Xc_{\rm f}\to L(\Zc_{\rm m})$ be operator-norm continuous and pointwise skew-adjoint.
For $x_{\rm f}\in\Xc_{\rm f}$, define
\[
U_{x_{\rm f}}
\left(
\begin{pmatrix}
f_0\\
f_{\rm m}
\end{pmatrix},
\begin{pmatrix}
e_0\\
e_{\rm m}
\end{pmatrix}
\right)
:=
\left(
\begin{pmatrix}
f_0\\
f_{\rm m}
\end{pmatrix},
\begin{pmatrix}
e_0\\
e_{\rm m}+K(x_{\rm f})f_{\rm m}
\end{pmatrix}
\right).
\]
Thus
\[
U_{x_{\rm f}}
\in
L\bigl((\Zc_0\times\Zc_{\rm m})^2\bigr).
\]
The transformation is invertible, with inverse obtained by replacing
$K(x_{\rm f})$ by $-K(x_{\rm f})$. Moreover,
\[
\langle\!\langle U_{x_{\rm f}}z_1,U_{x_{\rm f}}z_2\rangle\!\rangle
-
\langle\!\langle z_1,z_2\rangle\!\rangle
=
\langle f_{{\rm m},1},K(x_{\rm f})f_{{\rm m},2}\rangle_{\Zc_{\rm m}}
+
\langle f_{{\rm m},2},K(x_{\rm f})f_{{\rm m},1}\rangle_{\Zc_{\rm m}}
=
0.
\]
Thus $U_{x_{\rm f}}$ preserves the power pairing. Regarding
$\widehat{\Dc}$ as the constant modulated family
$\Dc_{x_{\rm f}}^0:=\widehat{\Dc}$ on $\Xc_{\rm f}$,
Proposition~\ref{prop:power-preserving-modulation} implies that
$\Dc_{\rm f}
=
\bigl(
\Dc_{{\rm f},x_{\rm f}}
\bigr)_{x_{\rm f}\in\Xc_{\rm f}}$,
$\Dc_{{\rm f},x_{\rm f}}
:=
U_{x_{\rm f}}\widehat{\Dc}$,
is a modulated Dirac structure. The differential operator and its domain
remain unchanged; all state dependence is confined to $\Zc_{\rm m}$.

\begin{remark}[Relation with finite-dimensional modulation]
\label{rem:automorphism-finite-dimensional-modulation}
Suppose that the ports intended for a subsequent interconnection form a
finite-dimensional coupling space $\Zc_{\rm c}$ and that
$U_{x_{\rm f}}$ acts only on the corresponding coupling variables. Let
$P_{\rm c}$ denote the projection onto all coupling-flow and coupling-effort
components and set
$\mathcal T
:=
\widehat{\Dc}\cap\ker P_{\rm c}$.
Then $\mathcal T$ is a closed power-isotropic subspace with vanishing
coupling components, and
\[
U_{x_{\rm f}}z=z
\qquad
\text{for all }z\in\mathcal T.
\]
Set
$r:=\dim(\widehat{\Dc}/\mathcal T)$, where $\widehat{\Dc}/\mathcal T$ denotes the
quotient space of $\widehat{\Dc}$ by $\mathcal T$. Since the restriction of $P_{\rm c}$ to $\widehat{\Dc}$ has kernel
$\mathcal T$, it induces an injective map
\[
\widehat{\Dc}/\mathcal T
\longrightarrow
\Zc_{\rm c}\times\Zc_{\rm c}.
\]
Hence $r<\infty$.
The power pairing induces a bilinear pairing
\[
\bigl(
\mathcal T^{\bot\!\!\!\bot}/\widehat{\Dc}
\bigr)
\times
\bigl(
\widehat{\Dc}/\mathcal T
\bigr)
\to\R,
\
([z],[d])
\mapsto
\langle\!\langle z,d\rangle\!\rangle.
\]
This pairing is well-defined. It is non-degenerate in the first argument
because an element of $\mathcal T^{\bot\!\!\!\bot}$ that is
power-orthogonal to $\widehat{\Dc}$ belongs to
$\widehat{\Dc}^{\bot\!\!\!\bot}=\widehat{\Dc}$. It is non-degenerate in
the second argument because an element of $\widehat{\Dc}$ that is
power-orthogonal to $\mathcal T^{\bot\!\!\!\bot}$ belongs to
\[
\bigl(
\mathcal T^{\bot\!\!\!\bot}
\bigr)^{\bot\!\!\!\bot}
=
\mathcal T,
\]
where we have used that $\mathcal T$ is closed. Non-degeneracy in the first argument induces an injective map from
$\mathcal T^{\bot\!\!\!\bot}/\widehat{\Dc}$ into the dual of the
$r$-dimensional space $\widehat{\Dc}/\mathcal T$. Hence
$\mathcal T^{\bot\!\!\!\bot}/\widehat{\Dc}$ is finite-dimensional and has
dimension at most $r$. Non-degeneracy in the second argument gives the
reverse inequality. Consequently,
\[
\dim\bigl(
\mathcal T^{\bot\!\!\!\bot}/\widehat{\Dc}
\bigr)
=
\dim\bigl(
\widehat{\Dc}/\mathcal T
\bigr)
=
r,
\]
and therefore
\[
\dim\bigl(
\mathcal T^{\bot\!\!\!\bot}/\mathcal T
\bigr)
=
2r.
\]
The induced power pairing on
$\mathcal T^{\bot\!\!\!\bot}/\mathcal T$
is non-degenerate. Indeed, if $z\in\mathcal T^{\bot\!\!\!\bot}$ is
power-orthogonal to all of $\mathcal T^{\bot\!\!\!\bot}$, then
\[
z\in
\bigl(
\mathcal T^{\bot\!\!\!\bot}
\bigr)^{\bot\!\!\!\bot}
=
\mathcal T,
\]
where the last equality follows from the closedness of $\mathcal T$.

Choose a finite-dimensional subspace $\mathcal W$ such that
\[
\mathcal T^{\bot\!\!\!\bot}
=
\mathcal T\oplus\mathcal W,
\qquad
\dim\mathcal W=2r,
\]
and let
$P_{\mathcal W}:
\mathcal T^{\bot\!\!\!\bot}\to\mathcal W$
be the bounded projection onto $\mathcal W$ along $\mathcal T$. The identification
\[
\mathcal T^{\bot\!\!\!\bot}/\mathcal T
\longrightarrow
\mathcal W,
\qquad
[z_{\mathcal T}+w]\mapsto w,\]
preserves the power pairing.

Choose an $r$-dimensional closed complement $\mathcal E^0$ of
$\mathcal T$ in $\widehat{\Dc}$, so that
$\widehat{\Dc}
=
\mathcal T\oplus\mathcal E^0$.
For $x_{\rm f}\in\Xc_{\rm f}$, define
\[
\mathcal E_{x_{\rm f}}
:=
P_{\mathcal W}
U_{x_{\rm f}}\mathcal E^0
\subset\mathcal W.
\]
Since $U_{x_{\rm f}}$ preserves the power pairing and fixes
$\mathcal T$, one has
\[
\mathcal T
\subset
U_{x_{\rm f}}\widehat{\Dc}
\subset
\mathcal T^{\bot\!\!\!\bot}
\;\text{ and }\;
U_{x_{\rm f}}\widehat{\Dc}
=
\mathcal T\oplus\mathcal E_{x_{\rm f}}.
\]

Moreover,
$U_{x_{\rm f}}\widehat{\Dc}/\mathcal T$
is self-orthogonal in
$\mathcal T^{\bot\!\!\!\bot}/\mathcal T$. Indeed, a class $[z]$ is
power-orthogonal to
$U_{x_{\rm f}}\widehat{\Dc}/\mathcal T$
if, and only if,
\[
z\in
\bigl(
U_{x_{\rm f}}\widehat{\Dc}
\bigr)^{\bot\!\!\!\bot}
=
U_{x_{\rm f}}\widehat{\Dc}.
\]
Under the power-preserving identification with $\mathcal W$, this gives
\[
\mathcal E_{x_{\rm f}}
=
\setdef{
w\in\mathcal W
}{
\langle\!\langle w,v\rangle\!\rangle=0
\text{ for all }v\in\mathcal E_{x_{\rm f}}
}.
\]
If $S^0:\R^r\to\mathcal E^0$ is an isomorphism, define
\[
S_{x_{\rm f}}
:=
P_{\mathcal W}U_{x_{\rm f}}S^0
:
\R^r\to\mathcal E_{x_{\rm f}}.
\]
This map is surjective by the definition of
$\mathcal E_{x_{\rm f}}$. If $S_{x_{\rm f}}v=0$, then
$U_{x_{\rm f}}S^0v\in\mathcal T$. Since $U_{x_{\rm f}}$ fixes
$\mathcal T$, it follows that
$S^0v\in\mathcal T\cap\mathcal E^0=\{0\}$, and hence $v=0$.
Thus $S_{x_{\rm f}}$ is an isomorphism, and the family depends
continuously on $x_{\rm f}$ in the operator norm. Consequently,
$\Dc_{\rm f}$ has a finite-dimensional modulation relative to
$\Zc_{\rm c}$ in the sense of
Definition~\ref{def:finite-dimensional-modulation}.\end{remark}

Additional finite-dimensional ports may be appended before applying the
modulation. In particular, for a finite-dimensional Hilbert space
$\Zc_{\rm free}$,
\[
\Dc_{\rm free}
:=
\setdef{
(f_{\rm free},e_{\rm free})
\in\Zc_{\rm free}\times\Zc_{\rm free}
}{
e_{\rm free}=0
}
\]
is a Dirac structure. Its flow is unrestricted and its effort vanishes.
Thus, a free finite-dimensional flow port can be added to
$\widehat{\Dc}$ and subsequently included in the modulation.

Selected finite-dimensional ports of $\Dc_{{\rm f},x_{\rm f}}$ may then be
interconnected with corresponding ports of another modulated Dirac
structure. If the coupling operator from
Proposition~\ref{prop:Dirac_comp} is surjective, the interconnected family is
again a modulated Dirac structure.

\begin{example}[Components of flexible multibody systems]\
\label{ex:expde}
\begin{enumerate}[label=(\alph*), ref=(\alph*)]
\item
\label{ex:expde1}
Consider a vibrating string; see \cite[Ex.~7.1.1]{JZ12}. Its transverse
displacement $w(\xi,t)$ satisfies
\begin{equation}
\tfrac{\partial^2 w}{\partial t^2}(\xi,t)
=
\tfrac{1}{\rho(\xi)}
\tfrac{\partial}{\partial\xi}
\left(
T(\xi)\tfrac{\partial w}{\partial\xi}(\xi,t)
\right),
\qquad
\xi\in[a,b],\quad t\geq0,
\label{eq:vibstr}
\end{equation}
where $\rho$ is the mass density and $T$ is the string tension. Introduce
the physical variables
\begin{align*}
p(\xi,t)
&=
\rho(\xi)\tfrac{\partial w}{\partial t}(\xi,t)
&&\text{transverse momentum density},
\\
\epsilon(\xi,t)
&=
\tfrac{\partial w}{\partial\xi}(\xi,t)
&&\text{strain}.
\end{align*}
With
$x
=
\begin{smallpmatrix}
p\\
\epsilon
\end{smallpmatrix}$,
 \eqref{eq:vibstr} is of the form~\eqref{eq:BH1gen} with
\[
N=1,
\;
D=0,
\;
B_0=0,
\;
P_0=0,
\quad
P_1=
\begin{smallbmatrix}
0&1\\
1&0
\end{smallbmatrix},
\;
H(\xi)=
\begin{smallbmatrix}
\rho(\xi)^{-1}&0\\
0&T(\xi)
\end{smallbmatrix}.
\]
A physical choice of boundary variables is
\[
f_\partial(t)
=
\begin{pmatrix}
(\rho^{-1}p)(a,t)\\
(\rho^{-1}p)(b,t)
\end{pmatrix},
\qquad
e_\partial(t)
=
\begin{pmatrix}
\phantom{-}(T\epsilon)(a,t)\\
-(T\epsilon)(b,t)
\end{pmatrix}.
\]
The boundary flows are the endpoint transverse velocities, whereas the
boundary efforts are the corresponding transverse reaction forces; see
\cite[eq.~(7.3)]{JZ12}.
\item
\label{ex:expde2}
Consider a Timoshenko beam; see \cite[Ex.~7.1.4]{JZ12}. Its transverse
displacement $w(\xi,t)$ and cross-section rotation $\phi(\xi,t)$ satisfy
\begin{align}
\rho(\xi)\tfrac{\partial^2 w}{\partial t^2}(\xi,t)
&=
\tfrac{\partial}{\partial\xi}
\left(
K(\xi)
\left(
\tfrac{\partial w}{\partial\xi}(\xi,t)-\phi(\xi,t)
\right)
\right),
\label{eq:beam_eq1}
\\
I_\rho(\xi)\tfrac{\partial^2\phi}{\partial t^2}(\xi,t)
&=
\tfrac{\partial}{\partial\xi}
\left(
EI(\xi)\tfrac{\partial\phi}{\partial\xi}(\xi,t)
\right)
+
K(\xi)
\left(
\tfrac{\partial w}{\partial\xi}(\xi,t)-\phi(\xi,t)
\right).
\label{eq:beam_eq2}
\end{align}
Here $\rho$ is the mass density, $I_\rho$ is the rotary inertia density,
$EI$ is the bending stiffness, and $K$ is the shear stiffness. Introduce
the physical variables
\begin{align*}
\gamma(\xi,t)
&=
\tfrac{\partial w}{\partial\xi}(\xi,t)-\phi(\xi,t)
&&\text{shear deformation},
\\
p(\xi,t)
&=
\rho(\xi)\tfrac{\partial w}{\partial t}(\xi,t)
&&\text{transverse momentum density},
\\
\theta(\xi,t)
&=
\tfrac{\partial\phi}{\partial\xi}(\xi,t)
&&\text{bending strain},
\\
M(\xi,t)
&=
I_\rho(\xi)\tfrac{\partial\phi}{\partial t}(\xi,t)
&&\text{angular momentum density}.
\end{align*}
With
$x
=
\begin{smallpmatrix}
\gamma\\
p\\
\theta\\
M
\end{smallpmatrix}$,
equations~\eqref{eq:beam_eq1}--\eqref{eq:beam_eq2} are of the
form~\eqref{eq:BH1gen} with
\[
N=1,
\;
D=0,
\;
B_0=0,
\quad
P_0=
\begin{smallbmatrix}
0&0&0&-1\\
0&0&0&0\\
0&0&0&0\\
1&0&0&0
\end{smallbmatrix},
\quad
P_1=
\begin{smallbmatrix}
0&1&0&0\\
1&0&0&0\\
0&0&0&1\\
0&0&1&0
\end{smallbmatrix},
\]
and
\[
H(\xi)
=
\begin{smallbmatrix}
K(\xi)&0&0&0\\
0&\rho(\xi)^{-1}&0&0\\
0&0&EI(\xi)&0\\
0&0&0&I_\rho(\xi)^{-1}
\end{smallbmatrix}.
\]
A physical choice of boundary variables is
\[
f_\partial(t)
=
\begin{pmatrix}
(\rho^{-1}p)(a,t)\\
(\rho^{-1}p)(b,t)\\
(I_\rho^{-1}M)(a,t)\\
(I_\rho^{-1}M)(b,t)
\end{pmatrix},
\qquad
e_\partial(t)
=
\begin{pmatrix}
\phantom{-}(K\gamma)(a,t)\\
-(K\gamma)(b,t)\\
\phantom{-}(EI\theta)(a,t)\\
-(EI\theta)(b,t)
\end{pmatrix}.
\]
The first two boundary flows are the endpoint transverse velocities and the
last two are the endpoint angular velocities. The corresponding boundary
efforts are the transverse reaction forces and reaction moments; see
\cite[eq.~(7.15)]{JZ12}.

\item
\label{ex:expde3}
Neglecting shear deformation and rotary inertia leads to the
Euler--Bernoulli beam; see \cite[Ex.~5.4]{Vill07}. Its transverse
displacement $w(\xi,t)$ satisfies
\begin{equation}
\rho(\xi)\tfrac{\partial^2 w}{\partial t^2}(\xi,t)
+
\tfrac{\partial^2}{\partial\xi^2}
\left(
EI(\xi)\tfrac{\partial^2 w}{\partial\xi^2}(\xi,t)
\right)
=
0,
\qquad
\xi\in[a,b],\quad t\geq0.
\label{eq:euler-bernoulli-reference}
\end{equation}
Here $\rho$ is the mass density and $EI$ is the bending stiffness.
Introduce the physical variables
\begin{align*}
\kappa(\xi,t)
&=
\tfrac{\partial^2 w}{\partial\xi^2}(\xi,t)
&&\text{curvature},
\\
p(\xi,t)
&=
\rho(\xi)\tfrac{\partial w}{\partial t}(\xi,t)
&&\text{transverse momentum density}.
\end{align*}
With
$x
=
\begin{smallpmatrix}
\kappa\\
p
\end{smallpmatrix}$,
equation~\eqref{eq:euler-bernoulli-reference} is of the
form~\eqref{eq:BH1gen} with
\[
N=2,
\;
D=0,
\;
B_0=0,
\;
P_0=P_1=0,
\quad
P_2=
\begin{smallbmatrix}
0&1\\
-1&0
\end{smallbmatrix},
\;
H(\xi)
=
\begin{smallbmatrix}
EI(\xi)&0\\
0&\rho(\xi)^{-1}
\end{smallbmatrix}.
\]
A physical choice of boundary variables is
\[
f_\partial(t)
=
\begin{pmatrix}
(\rho^{-1}p)(a,t)\\
(\rho^{-1}p)(b,t)\\
\tfrac{\partial}{\partial\xi}(\rho^{-1}p)(a,t)\\
\tfrac{\partial}{\partial\xi}(\rho^{-1}p)(b,t)
\end{pmatrix},
\qquad
e_\partial(t)
=
\begin{pmatrix}
-\tfrac{\partial}{\partial\xi}(EI\kappa)(a,t)\\
\phantom{-}\tfrac{\partial}{\partial\xi}(EI\kappa)(b,t)\\
\phantom{-}(EI\kappa)(a,t)\\
-(EI\kappa)(b,t)
\end{pmatrix}.
\]
The first two boundary flows are the endpoint transverse velocities and the
last two are the endpoint angular velocities. The corresponding boundary
efforts are the transverse reaction forces and reaction moments; see
\cite[Ex.~5.7.2]{AugnerDis}.
\end{enumerate}
\end{example}

\section{A moving flexible beam and a flexible slider--crank mechanism}
\label{sec:ex}
As an example of a flexible multibody system this section treats a moving flexible beam and its integration in a slider-crank mechanism. This is a typical example of a flexible multibody system with a kinematic loop. 
Throughout this section, we neglect damping in order to focus on the
power-preserving coupling between the rigid and flexible components.
Accordingly, the resistive spaces are chosen to be zero-dimensional and
the corresponding resistive relations are trivial. Damping can be included through nontrivial resistive relations as described
in Sections~\ref{sec:rigidmks} and~\ref{sec:flexiblemks}, without changing
the interconnection construction.

\subsection{A moving-frame Euler--Bernoulli beam}
\label{sec:moving-beam}

We consider an Euler--Bernoulli beam moving in the two-dimensional plane.
The deformation is described in the spirit of the classical floating-frame of reference approach, see e.g.~\cite{Shabana20}. In this study its deformation is described in a frame attached to the straight segment
joining its two endpoints, see Fig.~\ref{fig:FlexibleBeam}, representing a Chord-frame. Their positions are denoted by
$\bm q_0(t),\bm q_f(t)\in\R^2$. Translational forces
$\bm F_0(t),\bm F_f(t)\in\R^2$ and bending moments
$M_0(t),M_f(t)\in\R$ act at the endpoints. Corresponding coordinate frames attached to the cross sections of the end-points move with 
translational velocities $\bm v_0(t),\bm v_f(t)$ and angular velocities
$\omega_0(t),\omega_f(t)$, respectively. Gravity is omitted in order to focus on the coupling
between rigid motion and elastic deformation.

For time-dependent quantities, the dependence on $t$ is displayed when the
quantity is introduced or when it is needed for clarity; it is suppressed
when all quantities in a relation are evaluated at the same time.

The beam occupies the reference interval $[0,L]$. The bending stiffness is $EI$ and the mass per unit length is denoted by $\rho$. We assume that
$\rho,\rho^{-1},EI,(EI)^{-1}\in L^\infty([0,L];\R)$ and that $\rho$ and $EI$
are uniformly positive. Its total mass is
\[
m:=\int_0^L\rho(\xi)\,{\rm d}\xi.
\]
\begin{figure}[t]
    \centering
    \includegraphics[width=0.6\linewidth]{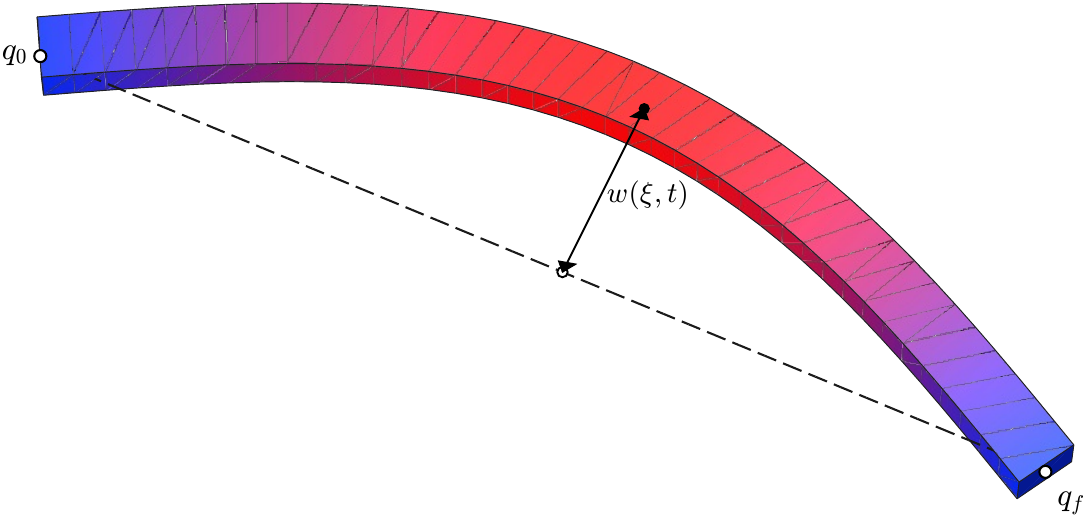}
    \caption{Moving planar Euler--Bernoulli beam and its endpoint ports.}
    \label{fig:FlexibleBeam}
\end{figure}
Related purely rotating-beam models can be found, for instance, in
\cite[Sec.~II.1.2.b]{Matt21} and~\cite{MattioniWuLeGorrec20}. Here,
translational and rotational ports are retained at both endpoints, so that
the beam can subsequently be interconnected with other multibody
components.

For
\[
U_{\rm pos}
:=
\setdef{
(\bm q_0,\bm q_f)\in\R^2\times\R^2
}{
\bm q_0\neq\bm q_f
},
\qquad
\bq_{\rm r}
:=
\begin{pmatrix}
\bm q_0\\
\bm q_f
\end{pmatrix},
\]
define
\[
\ell_{\rm r}(\bq_{\rm r})
:=
\|\bm q_f-\bm q_0\|_2,
\qquad
\mathtt t(\bq_{\rm r})
:=
\frac{\bm q_f-\bm q_0}{\ell_{\rm r}(\bq_{\rm r})},
\qquad
\mathtt n(\bq_{\rm r})
:=
J\mathtt t(\bq_{\rm r}),
\]
where
\[
J
:=
\begin{bmatrix}
0&-1\\
1&0
\end{bmatrix}.
\]
We omit the dependence on $\bq_{\rm r}$ whenever no ambiguity arises.

We choose the generalized velocities so that the endpoint velocities have
the decompositions
\[
\bm v_0
=
v_{\mathtt t}\mathtt t+v_{{\mathtt n}0}\mathtt n,
\qquad
\bm v_f
=
v_{\mathtt t}\mathtt t+v_{{\mathtt n}f}\mathtt n.
\]
Thus $v_{\mathtt t}$ is the common tangential endpoint velocity, while
$v_{{\mathtt n}0}$ and $v_{{\mathtt n}f}$ are the normal velocities of the
left and right endpoints, respectively.  These three velocities describe the velocity state of the reference frame. Set
\[
\bv_{\rm r}
:=
\begin{pmatrix}
v_{\mathtt t}\\
v_{{\mathtt n}0}\\
v_{{\mathtt n}f}
\end{pmatrix}.
\]
The endpoint kinematics are
\begin{equation}
\dot{\bq}_{\rm r}
=
Z(\bq_{\rm r})\bv_{\rm r},
\qquad
Z(\bq_{\rm r})
=
\begin{bmatrix}
\mathtt t&\mathtt n&0_{2\times1}\\
\mathtt t&0_{2\times1}&\mathtt n
\end{bmatrix},
\qquad
\omega
=
\frac{v_{{\mathtt n}f}-v_{{\mathtt n}0}}
{\ell_{\rm r}(\bq_{\rm r})}.
\label{eq:flexbeam-kinematics}
\end{equation}
Since
\[
\tfrac{{\rm d}}{{\rm d}t}\ell_{\rm r}(\bq_{\rm r})
=
\mathtt t^\top
\bigl(
\dot{\bm q}_f-\dot{\bm q}_0
\bigr)
=
0,
\]
the endpoint distance is constant. We restrict attention to initial
configurations satisfying
$\ell_{\rm r}(\bq_{\rm r}(0))=L$. Then
$\ell_{\rm r}(\bq_{\rm r}(t))=L$ along every corresponding trajectory and
\[
\omega
=
\frac{v_{{\mathtt n}f}-v_{{\mathtt n}0}}{L}.
\]
Let $w(\xi,t)$ denote the transverse displacement relative to the reference
segment. Along a trajectory, we write
$\mathtt t(t):=\mathtt t(\bq_{\rm r}(t))$ and
$\mathtt n(t):=\mathtt n(\bq_{\rm r}(t))$.
The physical centerline is
\begin{equation}
\bm r(\xi,t)
=
\bm q_0(t)+\xi\mathtt t(t)+w(\xi,t)\mathtt n(t),
\qquad
w(0,t)=w(L,t)=0.
\label{eq:flexbeam-centerline}
\end{equation}
As in the standard Euler--Bernoulli approximation, transverse deflections
and slopes are assumed to be small, while axial and shear deformation are
neglected. Using
$\dot{\mathtt t}=\omega\mathtt n$ and
$\dot{\mathtt n}=-\omega\mathtt t$, one obtains
\[
\tfrac{\partial\bm r}{\partial t}(\xi,t)
=
\bigl(
v_{\mathtt t}(t)-\omega(t)w(\xi,t)
\bigr)\mathtt t(t)
+
\left(
v_{{\mathtt n}0}(t)+\xi\omega(t)
+
\tfrac{\partial w}{\partial t}(\xi,t)
\right)\mathtt n(t).
\]
Consistently with the small-deflection approximation, we neglect the
deformation-dependent tangential velocity $-\omega w\mathtt t$, whose
magnitude is smaller than the characteristic rigid rotational velocity
$|\omega|L$ by the factor $|w|/L$. This approximation excludes centrifugal
stiffening and higher-order deformation-dependent inertial effects, while
retaining the leading moving-frame terms generated by the rotation of
$\mathtt t$ and $\mathtt n$. On the invariant set
$\ell_{\rm r}=L$, this gives
\[
\tfrac{\partial\bm r}{\partial t}(\xi,t)
\approx
v_{\mathtt t}(t)\mathtt t(t)
+
v_{\rm b}(\xi,t)\mathtt n(t),
\]
where
\[
v_{\rm b}(\xi,t)
=
\left(1-\tfrac{\xi}{L}\right)v_{{\mathtt n}0}(t)
+
\tfrac{\xi}{L}v_{{\mathtt n}f}(t)
+
\tfrac{\partial w}{\partial t}(\xi,t).
\]
Introduce the tangential and distributed transverse momenta
\[
p_{\mathtt t}
:=
mv_{\mathtt t},
\qquad
p_{\rm b}
:=
\rho v_{\rm b}.
\]
Under the linearized moving-frame kinematics, the kinetic energy is
\[
\mathcal T_{\rm b}
=
\tfrac12 m v_{\mathtt t}^2
+
\tfrac12
\int_0^L
\rho(\xi)
\left(
\left(1-\tfrac{\xi}{L}\right)v_{{\mathtt n}0}
+
\tfrac{\xi}{L}v_{{\mathtt n}f}
+
\tfrac{\partial w}{\partial t}(\xi)
\right)^2
\,{\rm d}\xi.
\]
In terms of the momenta introduced above, this becomes
\[
\mathcal T_{\rm b}
=
\frac{p_{\mathtt t}^2}{2m}
+
\frac12
\int_0^L
\rho(\xi)^{-1}p_{\rm b}(\xi)^2
\,{\rm d}\xi.
\]
The momenta conjugate to the normal endpoint velocities are the
partial derivatives of $\mathcal T_{\rm b}$ with respect to
$v_{{\mathtt n}0}$ and $v_{{\mathtt n}f}$, derived as
\[
\int_0^L
\left(1-\tfrac{\xi}{L}\right)p_{\rm b}(\xi)\,{\rm d}\xi
\qquad\text{and}\qquad
\int_0^L
\tfrac{\xi}{L}p_{\rm b}(\xi)\,{\rm d}\xi,
\]
respectively. Hence they are determined by the distributed momentum
$p_{\rm b}$ and do not constitute independent finite-dimensional storage
variables.

Introduce the curvature
\[
\kappa
:=
\frac{\partial^2w}{\partial\xi^2}.
\]
The total beam Hamiltonian, defined as the sum of the kinetic and bending
energies, is therefore
\begin{equation}
\mathcal H_{\rm b}
=
\tfrac{p_{\mathtt t}^2}{2m}
+
\tfrac12
\int_0^L
\left(
\rho(\xi)^{-1}p_{\rm b}(\xi)^2
+
EI(\xi)\kappa(\xi)^2
\right)
\,{\rm d}\xi.
\label{eq:flexbeam-H}
\end{equation}
Thus the choice of $p_{\rm b}$ absorbs the kinetic cross terms between the
normal endpoint motion and the relative deformation into the distributed
momentum. The corresponding moving-frame coupling terms reappear below as
finite-dimensional modulations of the Dirac structures of
the rigid and flexible subsystems.

We briefly derive the balance equations. Differentiating the linearized
velocity
$v_{\mathtt t}\mathtt t+v_{\rm b}\mathtt n$
and using
$\dot{\mathtt t}=\omega\mathtt n$ and
$\dot{\mathtt n}=-\omega\mathtt t$
gives the linearized acceleration
\[
\tfrac{\partial^2\bm r}{\partial t^2}
\approx
\bigl(
\dot v_{\mathtt t}-\omega v_{\rm b}
\bigr)\mathtt t
+
\left(
\tfrac{\partial v_{\rm b}}{\partial t}
+\omega v_{\mathtt t}
\right)\mathtt n.
\]
Integrating the tangential component of the linear-momentum balance over
$[0,L]$ yields
\[
m\dot v_{\mathtt t}
-
\omega\int_0^L\rho(\xi)v_{\rm b}(\xi)\,{\rm d}\xi
=
\mathtt t^\top\bm F_0+\mathtt t^\top\bm F_f.
\]
The normal component of the Euler--Bernoulli balance is
\[
\rho
\left(
\tfrac{\partial v_{\rm b}}{\partial t}
+\omega v_{\mathtt t}
\right)
=
-
\tfrac{\partial^2}{\partial\xi^2}(EI\kappa).
\]
Moreover, the affine interpolation of the endpoint velocities in
$v_{\rm b}$ has vanishing second spatial derivative. Hence
\[
\tfrac{\partial\kappa}{\partial t}
=
\tfrac{\partial^2}{\partial\xi^2}
\left(
\tfrac{\partial w}{\partial t}
\right)
=
\tfrac{\partial^2v_{\rm b}}{\partial\xi^2}.
\]
Using
$p_{\mathtt t}=mv_{\mathtt t}$,
$p_{\rm b}=\rho v_{\rm b}$, and the endpoint kinematics in
\eqref{eq:flexbeam-kinematics}, these identities give
\begin{equation}
\begin{aligned}
\dot{\bq}_{\rm r}
&=
Z(\bq_{\rm r})
\begin{pmatrix}
p_{\mathtt t}/m\\
(\rho^{-1}p_{\rm b})(0)\\
(\rho^{-1}p_{\rm b})(L)
\end{pmatrix},
\\
\dot p_{\mathtt t}
&=
\omega
\int_0^L p_{\rm b}(\xi)\,{\rm d}\xi
+
\mathtt t^\top\bm F_0
+
\mathtt t^\top\bm F_f,
\\
\tfrac{\partial\kappa}{\partial t}
&=
\tfrac{\partial^2}{\partial\xi^2}
\bigl(
\rho^{-1}p_{\rm b}
\bigr),
\\
\tfrac{\partial p_{\rm b}}{\partial t}
&=
-
\tfrac{\partial^2}{\partial\xi^2}
\bigl(
EI\kappa
\bigr)
-
\rho\omega\tfrac{p_{\mathtt t}}m,
\\
\omega
&=
\tfrac{
(\rho^{-1}p_{\rm b})(L)
-
(\rho^{-1}p_{\rm b})(0)
}{L}.
\end{aligned}
\label{eq:flexbeam-dynamics}
\end{equation}
For classical trajectories, for every fixed time $t$,
\[
EI\kappa(t),\rho^{-1}p_{\rm b}(t)\in H^2([0,L];\R).
\]
The force and moment boundary conditions are
\begin{equation}
\begin{aligned}
\tfrac{\partial}{\partial\xi}(EI\kappa)(0)
&=
\mathtt n^\top\bm F_0,
&
-
\tfrac{\partial}{\partial\xi}(EI\kappa)(L)
&=
\mathtt n^\top\bm F_f,
\\
-(EI\kappa)(0)
&=
M_0,
&
(EI\kappa)(L)
&=
M_f,
\end{aligned}
\label{eq:flexbeam-boundary-conditions}
\end{equation}
whereas the translational and angular endpoint outputs are
\begin{equation}
\begin{aligned}
\bm v_0
&=
\tfrac{p_{\mathtt t}}m\mathtt t
+
(\rho^{-1}p_{\rm b})(0)\mathtt n,
&
\bm v_f
&=
\tfrac{p_{\mathtt t}}m\mathtt t
+
(\rho^{-1}p_{\rm b})(L)\mathtt n,
\\
\omega_0
&=
\tfrac{\partial}{\partial\xi}
\bigl(
\rho^{-1}p_{\rm b}
\bigr)(0),
&
\omega_f
&=
\tfrac{\partial}{\partial\xi}
\bigl(
\rho^{-1}p_{\rm b}
\bigr)(L).
\end{aligned}
\label{eq:flexbeam-endpoint-outputs}
\end{equation}

We next derive a port-Hamiltonian decomposition of this model. The rigid and
flexible storage states are
\[
x_{\rm r}
=
\begin{pmatrix}
\bq_{\rm r}\\
p_{\mathtt t}
\end{pmatrix}
\in
\Xc_{\rm r}
:=
U_{\rm pos}\times\R,
\qquad
x_{\rm f}
=
\begin{pmatrix}
\kappa\\
p_{\rm b}
\end{pmatrix}
\in
\Xc_{\rm f}
:=
L^2([0,L];\R^2).
\]
The modulated Dirac structures of the rigid and flexible
subsystems will be defined on $\Xc_{\rm r}$ and $\Xc_{\rm f}$,
respectively. The mechanical beam model is their restriction to the invariant
subset $\ell_{\rm r}(\bq_{\rm r})=L$.

The energy-storage relations are
\begin{align*}
\Lc_{\rm r}
&:=
\setdef{
\left(
\begin{pmatrix}
\bq_{\rm r}\\
p_{\mathtt t}
\end{pmatrix},
\begin{pmatrix}
e_q^{\rm r}\\
e_p^{\rm r}
\end{pmatrix}
\right)
\in(\R^5)^2
}{
\bq_{\rm r}\in U_{\rm pos},
\quad
e_q^{\rm r}=0,
\quad
e_p^{\rm r}=p_{\mathtt t}/m
},
\\
\Lc_{\rm f}
&:=
\setdef{
\left(
\begin{pmatrix}
\kappa\\
p_{\rm b}
\end{pmatrix},
\begin{pmatrix}
e_\kappa\\
e_p
\end{pmatrix}
\right)
\in
\bigl(
L^2([0,L];\R^2)
\bigr)^2
}{
e_\kappa=EI\kappa,
\quad
e_p=\rho^{-1}p_{\rm b}
}.
\end{align*}
The rigid relation is the graph of the gradient of
$(\bq_{\rm r},p_{\mathtt t})\mapsto p_{\mathtt t}^2/(2m)$ on
$U_{\rm pos}\times\R$ and is therefore a Lagrangian submanifold by Proposition~\ref{prop:infdimlagr}.
The flexible relation is induced by the bounded self-adjoint multiplication
operator with diagonal entries $EI$ and $\rho^{-1}$ and is Lagrangian by
Proposition~\ref{prop:infdimlagr} as well.

The common coupling space is
$\Zc_{\rm c}
:=
\R^4$,
with coupling flows and efforts ordered as
\[
f_{\rm c}
=
\begin{pmatrix}
f_{\rm d}\\
f_{\mathtt t}\\
f_{{\mathtt n}0}\\
f_{{\mathtt n}f}
\end{pmatrix},
\qquad
e_{\rm c}
=
\begin{pmatrix}
e_{\rm d}\\
e_{\mathtt t}\\
e_{{\mathtt n}0}\\
e_{{\mathtt n}f}
\end{pmatrix}.
\]
The first port represents the distributed inertial coupling, the second is
the tangential port, and the last two are the normal endpoint ports.

For the rigid subsystem, set
$\Zc_{\rm r}
:=
\R^5\times\R^4$,
where the first factor contains the rigid storage variables and the second
the external translational port. Define
\[
B_{\rm ext}^{\rm r}(\bq_{\rm r})
:=
\begin{bmatrix}
\mathtt t^\top&\mathtt t^\top\\
\mathtt n^\top&0_{1\times2}\\
0_{1\times2}&\mathtt n^\top
\end{bmatrix}.
\]
The external translational effort is
\[
e_{\rm ext}^{\rm r}
=
-
\begin{pmatrix}
\bm F_0\\
\bm F_f
\end{pmatrix},
\]
in accordance with the generator convention.

For a prospective element of the Dirac structure of the
rigid subsystem, write
\[
f_{\rm c}^{\rm r}
=
\begin{pmatrix}
f_{\rm d}^{\rm r}\\
f_{\mathtt t}^{\rm r}\\
f_{{\mathtt n}0}^{\rm r}\\
f_{{\mathtt n}f}^{\rm r}
\end{pmatrix},
\qquad
e_{\rm c}^{\rm r}
=
\begin{pmatrix}
e_{\rm d}^{\rm r}\\
e_{\mathtt t}^{\rm r}\\
e_{{\mathtt n}0}^{\rm r}\\
e_{{\mathtt n}f}^{\rm r}
\end{pmatrix},
\]
and set
\[
v_{\rm r}
:=
\begin{pmatrix}
e_p^{\rm r}\\
f_{{\mathtt n}0}^{\rm r}\\
f_{{\mathtt n}f}^{\rm r}
\end{pmatrix},
\qquad
K_{\rm r}(p_{\mathtt t})
:=
\frac{p_{\mathtt t}}{mL}
\begin{bmatrix}
0&0&1&-1\\
0&0&0&0\\
-1&0&0&0\\
1&0&0&0
\end{bmatrix}.
\]
For
$x_{\rm r}=(\bq_{\rm r},p_{\mathtt t})\in\Xc_{\rm r}$, define
$\Dc_{{\rm r},x_{\rm r}}
\subset
(\Zc_{\rm r}\times\Zc_{\rm c})^2$
as the set of all flow and effort variables satisfying, for some
$\overline e_{\rm r}\in\R^3$ ordered according to the coupling ports
$(\mathtt t,{\mathtt n}0,{\mathtt n}f)$,
\begin{equation}
\begin{aligned}
f_q^{\rm r}
&=
Z(\bq_{\rm r})v_{\rm r},
\quad
f_{\rm ext}^{\rm r}
=
\bigl(B_{\rm ext}^{\rm r}(\bq_{\rm r})\bigr)^\top v_{\rm r},
\quad
f_{\mathtt t}^{\rm r}
=
e_p^{\rm r},
\\
0
&=
Z(\bq_{\rm r})^\top e_q^{\rm r}
+
\begin{pmatrix}
f_p^{\rm r}\\
0\\
0
\end{pmatrix}
+
B_{\rm ext}^{\rm r}(\bq_{\rm r})e_{\rm ext}^{\rm r}
+
\overline e_{\rm r},
\\
e_{\rm c}^{\rm r}
&=
\begin{pmatrix}
0\\
\overline e_{\rm r}
\end{pmatrix}
+
K_{\rm r}(p_{\mathtt t})f_{\rm c}^{\rm r}.
\end{aligned}
\label{eq:flexbeam-rigid-dirac}
\end{equation}
\begin{proposition}
\label{prop:flexbeam-rigid-dirac}
Let
$\Xc_{\rm r}=U_{\rm pos}\times\R$,
$\Zc_{\rm r}=\R^5\times\R^4$, and
$\Zc_{\rm c}=\R^4$. For each
$x_{\rm r}\in\Xc_{\rm r}$, let
$\Dc_{{\rm r},x_{\rm r}}
\subset
(\Zc_{\rm r}\times\Zc_{\rm c})^2$
be the relation defined by \eqref{eq:flexbeam-rigid-dirac}. Then
$\Dc_{\rm r}
:=
\bigl(
\Dc_{{\rm r},x_{\rm r}}
\bigr)_{x_{\rm r}\in\Xc_{\rm r}}$
is a modulated Dirac structure.
\end{proposition}

\begin{proof}
First set $p_{\mathtt t}=0$, so that the coupling modulation vanishes, and
denote the resulting relation by
$\Dc_{{\rm r},\bq_{\rm r}}^0$. The total power of an element of this
relation equals
\[
v_{\rm r}^\top
\left(
Z(\bq_{\rm r})^\top e_q^{\rm r}
+
\begin{pmatrix}
f_p^{\rm r}\\
0\\
0
\end{pmatrix}
+
B_{\rm ext}^{\rm r}(\bq_{\rm r})e_{\rm ext}^{\rm r}
+
\overline e_{\rm r}
\right)
=
0.
\]
Hence the total power vanishes for every element of
$\Dc_{{\rm r},\bq_{\rm r}}^0$. Since
$\Dc_{{\rm r},\bq_{\rm r}}^0$ is a linear subspace, for any two elements
$d=(f,e)$ and $\widehat d=(\widehat f,\widehat e)$ of this relation, their
sum also belongs to $\Dc_{{\rm r},\bq_{\rm r}}^0$. Consequently,
\[
0
=
\langle f+\widehat f,e+\widehat e\rangle
=
\langle f,\widehat e\rangle
+
\langle\widehat f,e\rangle,
\]
where the vanishing self-power terms have been omitted. Thus
$\Dc_{{\rm r},\bq_{\rm r}}^0$ is power-isotropic.

Its elements are uniquely determined by the freely chosen variables
\[
v_{\rm r}\in\R^3,
\qquad
f_{\rm d}^{\rm r}\in\R,
\qquad
e_q^{\rm r}\in\R^4,
\qquad
e_{\rm ext}^{\rm r}\in\R^4,
\]
and the tangential component of $\overline e_{\rm r}$. Hence
\[
\dim\Dc_{{\rm r},\bq_{\rm r}}^0
=
3+1+4+4+1
=
13,
\]
which equals the dimension of the underlying flow space
$\Zc_{\rm r}\times\Zc_{\rm c}$. Since the ambient power pairing is
non-degenerate, power-isotropy and this dimension identity imply
\[
\Dc_{{\rm r},\bq_{\rm r}}^0
=
\bigl(
\Dc_{{\rm r},\bq_{\rm r}}^0
\bigr)^{\bot\!\!\!\bot}.
\]
Thus $\Dc_{{\rm r},\bq_{\rm r}}^0$ is a Dirac structure.

The preceding parametrization depends continuously on $\bq_{\rm r}$ and
provides a global trivialization of the family
$\bigl(
\Dc_{{\rm r},\bq_{\rm r}}^0
\bigr)_{\bq_{\rm r}\in U_{\rm pos}}$.
Extending this family constantly in $p_{\mathtt t}$ gives a modulated Dirac
structure on $\Xc_{\rm r}$.

Finally, $K_{\rm r}(p_{\mathtt t})$ is skew-symmetric and depends
continuously on $p_{\mathtt t}$. The transformation that leaves all flows
and all non-coupling efforts unchanged and adds
$K_{\rm r}(p_{\mathtt t})f_{\rm c}^{\rm r}$ to the coupling efforts
therefore preserves the power pairing. The assertion follows from
Proposition~\ref{prop:power-preserving-modulation}.
\end{proof}
For the flexible subsystem, set
$\Zc_{\rm f}
:=
L^2([0,L];\R^2)\times\R^2$,
where the second factor contains the two external moment ports. Write
\[
f_{\rm M}^{\rm f}
=
\begin{pmatrix}
f_{{\rm M}0}^{\rm f}\\
f_{{\rm M}f}^{\rm f}
\end{pmatrix},
\qquad
e_{\rm M}^{\rm f}
=
\begin{pmatrix}
e_{{\rm M}0}^{\rm f}\\
e_{{\rm M}f}^{\rm f}
\end{pmatrix}
=
-
\begin{pmatrix}
M_0\\
M_f
\end{pmatrix}.
\]
Let
$\Dc_{\rm f}^0
\subset
(\Zc_{\rm f}\times\Zc_{\rm c})^2$
consist of all flow and effort variables satisfying
\begin{equation}
\begin{aligned}
f_\kappa
&=
\tfrac{\partial^2e_p}{\partial\xi^2},
&
f_p
&=
-
\tfrac{\partial^2e_\kappa}{\partial\xi^2}
-
\rho e_{\rm d}^{\rm f},
&
f_{\rm d}^{\rm f}
&=
\int_0^L\rho(\xi)e_p(\xi)\,{\rm d}\xi,
&
e_{\mathtt t}^{\rm f}
&=
0,
\\
f_{{\mathtt n}0}^{\rm f}
&=
e_p(0),
&
e_{{\mathtt n}0}^{\rm f}
&=
-
\tfrac{\partial e_\kappa}{\partial\xi}(0),
&
f_{{\mathtt n}f}^{\rm f}
&=
e_p(L),
&
e_{{\mathtt n}f}^{\rm f}
&=
\tfrac{\partial e_\kappa}{\partial\xi}(L),
\\
f_{{\rm M}0}^{\rm f}
&=
\tfrac{\partial e_p}{\partial\xi}(0),
&
e_{{\rm M}0}^{\rm f}
&=
e_\kappa(0),
&
f_{{\rm M}f}^{\rm f}
&=
\tfrac{\partial e_p}{\partial\xi}(L),
&
e_{{\rm M}f}^{\rm f}
&=
-e_\kappa(L),
\end{aligned}
\label{eq:flexbeam-flexible-dirac}
\end{equation}
where $e_\kappa,e_p\in H^2([0,L];\R)$, whereas
$e_{\rm d}^{\rm f}\in\R$ and $f_{\mathtt t}^{\rm f}\in\R$
are free port variables.
\begin{proposition}
\label{prop:flexbeam-flexible-dirac}
Let
$\Zc_{\rm f}=L^2([0,L];\R^2)\times\R^2$ and
$\Zc_{\rm c}=\R^4$, and let
$\Dc_{\rm f}^0
\subset
(\Zc_{\rm f}\times\Zc_{\rm c})^2$
be the relation defined by \eqref{eq:flexbeam-flexible-dirac}. Then
$\Dc_{\rm f}^0$ is a Dirac structure.
\end{proposition}

\begin{proof}
The relations in \eqref{eq:flexbeam-flexible-dirac} for the distributed
storage flows $(f_\kappa,f_p)$, the distributed coupling port
$(f_{\rm d}^{\rm f},e_{\rm d}^{\rm f})$, and the normal and moment endpoint
ports are obtained from Proposition~\ref{Prop:Diracflex} and
Example~\ref{ex:expde}\,\ref{ex:expde3}, with vanishing damping and
\[
B_0(\xi)
=
\begin{pmatrix}
0\\
\rho(\xi)
\end{pmatrix}.
\]
The tangential port is the product with the finite-dimensional Dirac structure
\[
\Dc_{\mathtt t}^{\rm free}
:=
\setdef{
(f_{\mathtt t},e_{\mathtt t})\in\R\times\R
}{
e_{\mathtt t}=0
}.
\]
A fixed permutation places the storage variables, the two external moment
ports, and the four coupling ports in the order used above. Products and
power-pairing-preserving permutations preserve the Dirac property. Hence
$\Dc_{\rm f}^0$ is a Dirac structure.
\end{proof}

Define the bounded linear functional
\[
P_{\rm b}:\Xc_{\rm f}\to\R,\
x_{\rm f}
=
\begin{pmatrix}
\kappa\\
p_{\rm b}
\end{pmatrix}
\mapsto
\int_0^L p_{\rm b}(\xi)\,{\rm d}\xi.
\]
On the mechanical coupling ports ordered as
$(\mathtt t,{\mathtt n}0,{\mathtt n}f)$, set
\[
K_{\rm f}(x_{\rm f})
:=
\frac{P_{\rm b}(x_{\rm f})}{L}
\begin{bmatrix}
0&-1&1\\
1&0&0\\
-1&0&0
\end{bmatrix}.
\]
Relative to the full coupling-port ordering
$({\rm d},\mathtt t,{\mathtt n}0,{\mathtt n}f)$, define
\[
\widetilde K_{\rm f}(x_{\rm f})
:=
\begin{bmatrix}
0&0_{1\times3}\\
0_{3\times1}&K_{\rm f}(x_{\rm f})
\end{bmatrix}.
\]
Let $U_{x_{\rm f}}^{\rm f}$ leave all flows and all non-coupling efforts
unchanged and transform the coupling efforts according to
\[
e_{\rm c}^{\rm f}
\longmapsto
e_{\rm c}^{\rm f}
+
\widetilde K_{\rm f}(x_{\rm f})f_{\rm c}^{\rm f}.
\]
For $x_{\rm f}\in\Xc_{\rm f}$, define
\[
\Dc_{{\rm f},x_{\rm f}}
:=
U_{x_{\rm f}}^{\rm f}\Dc_{\rm f}^0,
\]
where $\Dc_{\rm f}^0$ is the Dirac structure from
Proposition~\ref{prop:flexbeam-flexible-dirac}.

\begin{proposition}
\label{prop:flexbeam-flexible-modulation}
The family
\[
\Dc_{\rm f}
:=
\bigl(
\Dc_{{\rm f},x_{\rm f}}
\bigr)_{x_{\rm f}\in\Xc_{\rm f}}
\]
is a modulated Dirac structure. Moreover, it has a finite-dimensional
modulation relative to $\Zc_{\rm c}$.
\end{proposition}
\begin{proof}
The matrix $\widetilde K_{\rm f}(x_{\rm f})$ is skew-symmetric for every
$x_{\rm f}\in\Xc_{\rm f}$. Since $P_{\rm b}$ is a bounded linear
functional, the mapping
$x_{\rm f}
\longmapsto
U_{x_{\rm f}}^{\rm f}$
is operator-norm continuous.
Proposition~\ref{prop:power-preserving-modulation} therefore implies that
$\Dc_{\rm f}$ is a modulated Dirac structure. Since
$U_{x_{\rm f}}^{\rm f}$ acts only on the finite-dimensional coupling
variables, the final assertion follows from
Remark~\ref{rem:automorphism-finite-dimensional-modulation}.
\end{proof}

The modulated flexible coupling efforts are explicitly
\begin{equation}
e_{\mathtt t}^{\rm f}
=
\tfrac{P_{\rm b}(x_{\rm f})}{L}
\left(
f_{{\mathtt n}f}^{\rm f}
-
f_{{\mathtt n}0}^{\rm f}
\right),
\;\;\;
e_{{\mathtt n}0}^{\rm f}
=
-
\tfrac{\partial e_\kappa}{\partial\xi}(0)
+
\tfrac{P_{\rm b}(x_{\rm f})}{L}f_{\mathtt t}^{\rm f},
\;\;\;
e_{{\mathtt n}f}^{\rm f}
=
\tfrac{\partial e_\kappa}{\partial\xi}(L)
-
\tfrac{P_{\rm b}(x_{\rm f})}{L}f_{\mathtt t}^{\rm f}.
\label{eq:flexbeam-flexible-modulated-efforts}
\end{equation}
The distributed coupling effort remains unrestricted and is unchanged by
the modulation.

The rigid and flexible coupling ports are interconnected by
$f_{\rm c}^{\rm r}
=
f_{\rm c}^{\rm f}$,
$e_{\rm c}^{\rm r}
=
-e_{\rm c}^{\rm f}$.
\begin{proposition}
\label{prop:flexbeam-interconnection}
Let $\Dc_{\rm r}$ be the modulated Dirac structure of the
rigid subsystem from
Proposition~\ref{prop:flexbeam-rigid-dirac}, and let $\Dc_{\rm f}$ be the
modulated Dirac structure of the flexible subsystem from
Proposition~\ref{prop:flexbeam-flexible-modulation}. For
$(x_{\rm r},x_{\rm f})\in\Xc_{\rm r}\times\Xc_{\rm f}$, define
\[
\Gamma_{x_{\rm r},x_{\rm f}}:
\Dc_{{\rm r},x_{\rm r}}
\times
\Dc_{{\rm f},x_{\rm f}}
\to
\Zc_{\rm c}\times\Zc_{\rm c},
\quad
(d_{\rm r},d_{\rm f})
\mapsto
\begin{pmatrix}
f_{\rm c}^{\rm r}-f_{\rm c}^{\rm f}\\
e_{\rm c}^{\rm r}+e_{\rm c}^{\rm f}
\end{pmatrix}.
\]
Then $\Gamma_{x_{\rm r},x_{\rm f}}$ is surjective for every
$(x_{\rm r},x_{\rm f})\in\Xc_{\rm r}\times\Xc_{\rm f}$.
Consequently,
\[
\Dc_{\rm b}
:=
\Pi
\bigl(
\Dc_{\rm r}
\circ_{\Zc_{\rm c}}
\Dc_{\rm f}
\bigr)
\]
is a modulated Dirac structure, where $\Pi$ is the fixed
power-pairing-preserving permutation that groups the rigid and flexible
storage variables and the remaining external ports. Together with
\[
\Lc_{\rm b}
:=
\Lc_{\rm r}\times\Lc_{\rm f}
\]
and the trivial resistive relation, this defines a port-Hamiltonian system.
On the invariant subset
$\ell_{\rm r}(\bq_{\rm r})=L$, this system is precisely the moving-beam
model \eqref{eq:flexbeam-dynamics}--\eqref{eq:flexbeam-endpoint-outputs}
with storage function \eqref{eq:flexbeam-H}.
\end{proposition}

\begin{proof}
Let $a,b\in\Zc_{\rm c}$ be arbitrary. The four rigid coupling flows may be
prescribed independently, while the flexible coupling flows can be chosen
to vanish by taking $e_p=0$ and $f_{\mathtt t}^{\rm f}=0$. Hence the flow
difference can be chosen equal to $a$.

The three mechanical components of the rigid coupling effort can be
prescribed by choosing $\overline e_{\rm r}$ appropriately, after
compensating the already determined contribution
$K_{\rm r}(p_{\mathtt t})f_{\rm c}^{\rm r}$. The rigid balance equation in
\eqref{eq:flexbeam-rigid-dirac} can then be satisfied because
$B_{\rm ext}^{\rm r}(\bq_{\rm r})$ has full row rank. The unrestricted
distributed effort of the flexible subsystem supplies the remaining
component of the effort sum. Thus
\[
\operatorname{im}\Gamma_{x_{\rm r},x_{\rm f}}
=
\Zc_{\rm c}\times\Zc_{\rm c}.
\]
Proposition~\ref{prop:Dirac_comp} now yields the asserted modulated Dirac
structure.

Substituting the storage relations, the external efforts
\[
e_{\rm ext}^{\rm r}
=
-
\begin{pmatrix}
\bm F_0\\
\bm F_f
\end{pmatrix},
\qquad
e_{\rm M}^{\rm f}
=
-
\begin{pmatrix}
M_0\\
M_f
\end{pmatrix},
\]
and the coupling conditions into the interconnected relation gives, on
$\ell_{\rm r}(\bq_{\rm r})=L$, the equations
\eqref{eq:flexbeam-dynamics}, the boundary conditions
\eqref{eq:flexbeam-boundary-conditions}, and the outputs
\eqref{eq:flexbeam-endpoint-outputs}. The storage function is the sum of the
rigid and flexible storage functions and therefore equals
\eqref{eq:flexbeam-H}.
\end{proof}

The two moving-frame contributions are assigned to different subsystem
Dirac structures. The rigid modulation gives
$e_{\rm d}^{\rm f}
=
\omega\,\tfrac{p_{\mathtt t}}m$
and therefore generates the distributed inertial term
$-\rho\omega p_{\mathtt t}/m$. The flexible modulation gives
\[
e_{\mathtt t}^{\rm f}
=
\omega P_{\rm b}(x_{\rm f})
=
\omega
\int_0^L p_{\rm b}(\xi)\,{\rm d}\xi
\]
and therefore generates the tangential reaction term in the equation for
$p_{\mathtt t}$. Thus each subsystem Dirac structure is modulated
exclusively by its own storage state.

\begin{remark}[Reconstruction of the relative displacement]
For every $\kappa\in L^2([0,L];\R)$, the relative displacement is recovered as
the unique solution of
\[
\tfrac{\partial^2w}{\partial\xi^2}
=
\kappa,
\qquad
w(0)=w(L)=0.
\]
It is given by
\[
w(\xi)
=
\int_0^\xi
(\xi-\eta)\kappa(\eta)\,{\rm d}\eta
-
\frac{\xi}{L}
\int_0^L
(L-\eta)\kappa(\eta)\,{\rm d}\eta.
\]
\end{remark}

\subsection{A slider--crank mechanism with a flexible connecting member}
\label{sec:slider-crank}

We now use the moving beam from Section~\ref{sec:moving-beam} as the
connecting member of a planar slider--crank mechanism; see
Figure~\ref{fig:flexible-slider-crank}. Its left endpoint is attached by an
ideal pin joint to the crank pin of a rigid wheel, whereas its right endpoint
moves along a fixed straight guide. No separate slider body or concentrated
slider mass is introduced. The right beam endpoint itself acts as the slider,
while the inertia and deformation of the connecting member are described by
the distributed state $(\kappa,p_{\rm b})$.

\begin{figure}[t]
  \centering
  \includegraphics[width=0.6\linewidth]{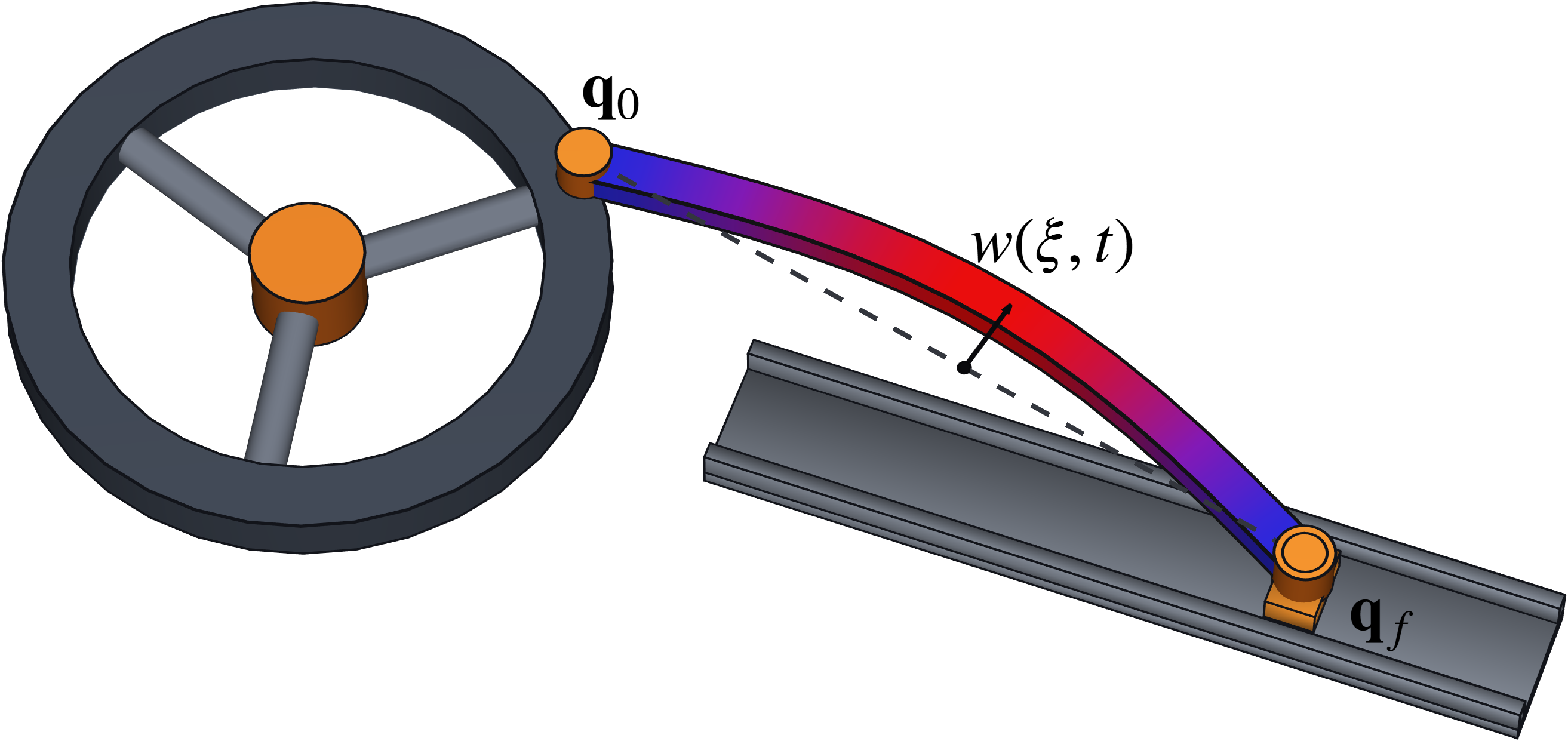}
  \caption{Planar slider--crank mechanism with a flexible connecting
  member. The left beam endpoint is attached to the crank pin of the rigid
  wheel, whereas the right endpoint moves along the fixed straight guide.}
  \label{fig:flexible-slider-crank}
\end{figure}

Let $\bm q_{\rm O}\in\R^2$ be the fixed wheel center, let $r_{\rm c}>0$ be
the crank radius, and let $\theta$ and $p_\theta$ denote the wheel angle and
angular momentum. With
$\bm a(\theta):=(\cos\theta,\sin\theta)^\top$ and
$\bm a_\perp(\theta):=J\bm a(\theta)$, the crank-pin position is
$\bm q_{\rm c}(\theta)=\bm q_{\rm O}+r_{\rm c}\bm a(\theta)$.
For a wheel inertia $I_{\rm w}>0$, its Hamiltonian and storage relation are
\[
\mathcal H_{\rm w}(\theta,p_\theta)
=
\frac{p_\theta^2}{2I_{\rm w}},
\qquad
\Lc_{\rm w}
=
\setdef{
\left(
\begin{pmatrix}
\theta\\
p_\theta
\end{pmatrix},
\begin{pmatrix}
0\\
p_\theta/I_{\rm w}
\end{pmatrix}
\right)
}{
(\theta,p_\theta)\in\R^2
}.
\]
The wheel has a two-dimensional crank-pin port
$(f_{\rm c}^{\rm w},e_{\rm c}^{\rm w})$ and a scalar driving-torque port
$(f_\tau^{\rm w},e_\tau^{\rm w})$. For each $(\theta,p_\theta)\in\R^2$, define
\begin{equation}
\Dc_{{\rm w},(\theta,p_\theta)}
:=
\setdef{
\left(
\begin{pmatrix}
f_\theta^{\rm w}\\
f_p^{\rm w}\\
f_\tau^{\rm w}\\
f_{\rm c}^{\rm w}
\end{pmatrix},
\begin{pmatrix}
e_\theta^{\rm w}\\
e_p^{\rm w}\\
e_\tau^{\rm w}\\
e_{\rm c}^{\rm w}
\end{pmatrix}
\right)
\in(\R^5)^2
}{
\begin{aligned}
f_\theta^{\rm w}
&=
e_p^{\rm w},
\;\;
f_\tau^{\rm w}
=
e_p^{\rm w},
\\
f_{\rm c}^{\rm w}
&=
r_{\rm c}\bm a_\perp(\theta)e_p^{\rm w},
\\
0
&=
f_p^{\rm w}
+
e_\theta^{\rm w}
+
e_\tau^{\rm w}
+
r_{\rm c}\bm a_\perp(\theta)^\top e_{\rm c}^{\rm w}
\end{aligned}
}.
\label{eq:slider-crank-wheel-dirac}
\end{equation}
For every $(\theta,p_\theta)\in\R^2$, this relation is power-isotropic and
five-dimensional, and hence a Dirac structure. Since its parametrization
depends smoothly on $\theta$ and is independent of $p_\theta$, the family
$\Dc_{\rm w}
:=
\bigl(
\Dc_{{\rm w},(\theta,p_\theta)}
\bigr)_{(\theta,p_\theta)\in\R^2}$
is a finite-dimensional modulated Dirac structure. If $\tau_{\rm w}$ is the physical torque acting on the wheel, then the
generator convention gives $e_\tau^{\rm w}=-\tau_{\rm w}$.

We denote the translational beam-endpoint ports by
$(f_0^{\rm b},e_0^{\rm b})$ and $(f_f^{\rm b},e_f^{\rm b})$, and the
corresponding moment ports by
$(f_{{\rm M}0}^{\rm b},e_{{\rm M}0}^{\rm b})$ and
$(f_{{\rm M}f}^{\rm b},e_{{\rm M}f}^{\rm b})$. Thus
$f_0^{\rm b}=\bm v_0$, $f_f^{\rm b}=\bm v_f$,
$e_0^{\rm b}=-\bm F_0$, and $e_f^{\rm b}=-\bm F_f$.
The crank-pin attachment is obtained by interconnecting the wheel port with
the left translational beam port. The pin joints transmit no bending
moments, so that
$e_{{\rm M}0}^{\rm b}=e_{{\rm M}f}^{\rm b}=0$, while the corresponding
angular-velocity flows remain unrestricted.

To describe the guide, let $\bm q_{\rm S}\in\R^2$ be a fixed point on it
and let $\bm a_{\rm s}\in\R^2$ be its unit direction. Set
$\bm n_{\rm s}:=J\bm a_{\rm s}$ and
$Q_{\rm s}:=(\bm a_{\rm s},\bm n_{\rm s})^\top$. The right translational
port is transformed according to
\begin{equation}
\begin{pmatrix}
f_{\rm s}^{\rm b}\\
f_{\rm n}^{\rm b}
\end{pmatrix}
=
Q_{\rm s}f_f^{\rm b},
\qquad
\begin{pmatrix}
e_{\rm s}^{\rm b}\\
e_{\rm n}^{\rm b}
\end{pmatrix}
=
Q_{\rm s}e_f^{\rm b}.
\label{eq:slider-crank-guide-transformation}
\end{equation}
Since $Q_{\rm s}$ is orthogonal, this transformation preserves the power
pairing. The parallel port
$(f_{\rm s}^{\rm b},e_{\rm s}^{\rm b})$ remains external and constitutes
the slider port, whereas the ideal guide imposes $f_{\rm n}^{\rm b}=0$ and
leaves the normal reaction effort unrestricted.

These conditions can be imposed simultaneously. Introduce the constant
termination structures
\[
\Dc_{\rm pin}
:=
\setdef{
(f_{\rm M},e_{\rm M})\in\R^2\times\R^2
}{
e_{\rm M}=0
},
\qquad
\Dc_{\rm guide}
:=
\setdef{
(f_{\rm n},e_{\rm n})\in\R\times\R
}{
f_{\rm n}=0
}.
\]
Let $\Pi_{\rm aux}$ be the fixed power-pairing-preserving permutation that
groups the crank-pin port, the two pin-moment ports, and the guide-normal
port into the coupling space, while leaving the wheel storage variables and
the driving-torque port among the remaining variables. Set
\[
\Dc_{\rm aux}
:=
\Pi_{\rm aux}
\bigl(
\Dc_{\rm w}\times\Dc_{\rm pin}\times\Dc_{\rm guide}
\bigr).
\]
This is a finite-dimensional modulated Dirac structure whose coupling ports
are the crank-pin port, the two pin-moment ports, and the guide-normal port. Let $\widetilde\Dc_{\rm b}$ denote the modulated Dirac
structure of the moving-beam subsystem after the fixed transformation \eqref{eq:slider-crank-guide-transformation} and a
power-pairing-preserving permutation that groups the crank-pin, moment, and
guide-normal ports into the coupling space
\[
\Zc_{\rm c}^{\rm sc}
=
\R^2\times\R\times\R\times\R.
\]
Let $\Pi_{\rm sc}$ be the fixed power-pairing-preserving permutation that
groups the wheel and beam storage variables and the remaining external
ports. The complete slider--crank interconnection is then
\begin{equation}
\Dc_{\rm sc}
:=
\Pi_{\rm sc}
\bigl(
\widetilde\Dc_{\rm b}
\circ_{\Zc_{\rm c}^{\rm sc}}
\Dc_{\rm aux}
\bigr).
\label{eq:slider-crank-dirac}
\end{equation}
We restrict the beam configuration to
\[
U_{\rm pos}^{\rm sc}
:=
\setdef{
\bq_{\rm r}\in U_{\rm pos}
}{
\bm a_{\rm s}^\top\mathtt t(\bq_{\rm r})\neq0
}.
\]
On this set, the coupling operator associated with
\eqref{eq:slider-crank-dirac} is surjective. Indeed,
\[
f_0^{\rm b}
=
\mathtt t\,e_p^{\rm r}
+
\mathtt n\,e_p(0)
\]
can be prescribed arbitrarily because $(\mathtt t,\mathtt n)$ is an
orthonormal basis. Once this flow has been fixed, the normal right-endpoint
flow
\[
f_{\rm n}^{\rm b}
=
\bm n_{\rm s}^\top\mathtt t\,e_p^{\rm r}
+
\bm a_{\rm s}^\top\mathtt t\,e_p(L)
\]
can be prescribed independently through $e_p(L)$. The two pin flows are
unrestricted. On the effort side, the crank-pin effort can be chosen
arbitrarily because the wheel balance can be satisfied through
$f_p^{\rm w}$, the two moment efforts can be prescribed independently
through the endpoint traces of $e_\kappa\in H^2([0,L];\R)$, and the guide
reaction effort is unrestricted. Proposition~\ref{prop:Dirac_comp}
therefore implies that $\Dc_{\rm sc}$ is a modulated Dirac structure.

By Remark~\ref{rem:order-of-interconnection}, relation
\eqref{eq:slider-crank-dirac} agrees, up to the ordering of the remaining
variables, with any successive construction in which the wheel, the two
pin terminations, and the guide are interconnected with the beam one after
another. Thus the mechanism is obtained from one finite-dimensional
modulated component and the moving beam by a single interconnection, as
described in
Remark~\ref{rem:several-finite-dimensional-subsystems}.

The total storage relation and Hamiltonian are
\[
\Lc_{\rm sc}
=
\Lc_{\rm w}\times\Lc_{\rm b},
\qquad
\mathcal H_{\rm sc}
=
\mathcal H_{\rm w}+\mathcal H_{\rm b}
=
\frac{p_\theta^2}{2I_{\rm w}}+\mathcal H_{\rm b}.
\]
The remaining external ports are the wheel-driving port
$(f_\tau^{\rm w},e_\tau^{\rm w})$ and the slider port
$(f_{\rm s}^{\rm b},e_{\rm s}^{\rm b})$.

If the initial configuration satisfies
$\bm q_0=\bm q_{\rm c}(\theta)$ and
$\bm n_{\rm s}^\top(\bm q_f-\bm q_{\rm S})=0$, then the crank-pin and guide
interconnections preserve these relations. Hence
$\bm q_f=\bm q_{\rm S}+s\bm a_{\rm s}$ with
$\dot s=f_{\rm s}^{\rm b}$, and together with
$\ell_{\rm r}=L$ this recovers the usual slider--crank geometry.

\section*{Statements and Declarations}

\subsection*{Funding}

This work was supported by the Deutsche Forschungsgemeinschaft (DFG, German
Research Foundation) through the project `\textit{Adaptive control of coupled rigid and
flexible multibody systems with port-Hamiltonian structure}' (Project-ID 362536361).

\subsection*{Author Contributions}

All authors contributed equally to this work.

\subsection*{Competing Interests}

The authors report no conflicts of interest.

\bibliography{pH-Multibody}

\end{document}